\documentclass[11pt]{article}
\usepackage{bcc23}
\usepackage{pstricks}
\usepackage{graphicx, pst-plot}

\newcommand{\ben}{\begin{enumerate}}
\newcommand{\een}{\end{enumerate}}
\newcommand{\ble}{\begin{lem}}
\newcommand{\ele}{\end{lem}}
\newcommand{\bth}{\begin{theorem}}
\renewcommand{\eth}{\end{theorem}}
\newcommand{\bpr}{\begin{prop}}
\newcommand{\epr}{\end{prop}}
\newcommand{\bco}{\begin{cor}}
\newcommand{\eco}{\end{cor}}
\newcommand{\bcon}{\begin{conj}}
\newcommand{\econ}{\end{conj}}
\newcommand{\bde}{\begin{defn}}
\newcommand{\ede}{\end{defn}}
\newcommand{\bex}{\begin{exa}}
\newcommand{\eex}{\end{exa}}
\newcommand{\barr}{\begin{array}}
\newcommand{\earr}{\end{array}}
\newcommand{\btab}{\begin{tabular}}
\newcommand{\etab}{\end{tabular}}
\newcommand{\beq}{\begin{equation}}
\newcommand{\eeq}{\end{equation}}
\newcommand{\bea}{\begin{eqnarray*}}
\newcommand{\eea}{\end{eqnarray*}}
\newcommand{\bce}{\begin{center}}
\newcommand{\ece}{\end{center}}
\newcommand{\bpi}{\begin{picture}}
\newcommand{\epi}{\end{picture}}
\newcommand{\bpp}{\begin{picture}}
\newcommand{\epp}{\end{picture}}
\newcommand{\bfi}{\begin{figure} \begin{center}}
\newcommand{\efi}{\end{center} \end{figure}}
\newcommand{\bprf}{\begin{proof}}
\newcommand{\eprf}{\end{proof}\medskip}
\newcommand{\capt}{\caption}
\newcommand{\bsl}{\begin{slide}{}}
\newcommand{\esl}{\end{slide}}
\newcommand{\bfr}{\begin{frame}}
\newcommand{\efr}{\end{frame}}

\newcommand{\qed}{\rule{1ex}{1ex}}
\newcommand{\hqed}{\hfill$\square$}

\newcommand{\eqed}[1]{$\textcolor{white}{\qed}\hfill{\dil#1}\hfill\square$}

\newcommand{\ol}{\overline}

\newcommand{\hs}[1]{\hspace{#1}}
\newcommand{\hso}[1]{\hspace{-1pt}}
\newcommand{\vs}[1]{\vspace{#1}}

\newcommand{\emp}{\emptyset}

\newcommand{\sbe}{\subseteq}

\newcommand{\setm}{\setminus}

\newcommand{\iso}{\cong}
\newcommand{\Cong}{\equiv}

\newcommand{\ptn}{\vdash}

\newcommand{\bdy}{\partial}

\newcommand{\case}[4]{\left\{\barr{ll}#1&\mbox{#2}\\#3&\mbox{#4}\earr\right.}

\newcommand{\ce}[1]{\lceil #1 \rceil}

\newcommand{\gau}[2]{\left[ \barr{c} #1 \\ #2 \earr \right]}

\def\<{\langle}
\def\>{\rangle}

\newcommand{\spn}[1]{\langle{#1}\rangle}
\newcommand{\ree}[1]{(\ref{#1})}

\newcommand{\ra}{\rightarrow}

\newcommand{\llra}{\longleftrightarrow}

\newcommand{\al}{\alpha}
\newcommand{\be}{\beta}
\newcommand{\ga}{\gamma}

\newcommand{\ep}{\epsilon}

\newcommand{\io}{\iota}

\newcommand{\la}{\lambda}

\newcommand{\om}{\omega}

\newcommand{\si}{\sigma}

\newcommand{\Ga}{\Gamma}
\newcommand{\De}{\Delta}

\newcommand{\Om}{\Omega}

\newcommand{\1}{{\bf 1}}
\newcommand{\2}{{\bf 2}}
\newcommand{\3}{{\bf 3}}
\newcommand{\4}{{\bf 4}}
\newcommand{\5}{{\bf 5}}
\newcommand{\6}{{\bf 6}}

\newcommand{\ba}{{\bf a}}
\newcommand{\bb}{{\bf b}}
\newcommand{\bc}{{\bf c}}

\newcommand{\bi}{{\bf i}}

\newcommand{\bs}{{\bf s}}

\newcommand{\bv}{{\bf v}}
\newcommand{\bw}{{\bf w}}

\newcommand{\bbC}{{\mathbb C}}

\newcommand{\bbN}{{\mathbb N}}
\newcommand{\bbP}{{\mathbb P}}
\newcommand{\bbQ}{{\mathbb Q}}
\newcommand{\bbR}{{\mathbb R}}

\newcommand{\bbZ}{{\mathbb Z}}

\newcommand{\cO}{{\cal O}}
\newcommand{\cP}{{\cal P}}

\newcommand{\cS}{{\cal S}}
\newcommand{\cT}{{\cal T}}

\newcommand{\fS}{{\mathfrak S}}

\newcommand{\cb}{\ol{c}}

\newcommand{\Hb}{\ol{H}}

\newcommand{\kb}{\ol{k}}

\newcommand{\ctt}{\operatorname{ct}}

\newcommand{\des}{\operatorname{des}}
\newcommand{\Des}{\operatorname{Des}}
\newcommand{\dia}{\operatorname{diag}}

\newcommand{\GL}{\operatorname{GL}}

\newcommand{\Hilb}{\operatorname{Hilb}}

\newcommand{\inv}{\operatorname{inv}}
\newcommand{\Inv}{\operatorname{Inv}}

\newcommand{\Mat}{\operatorname{Mat}}
\newcommand{\maj}{\operatorname{maj}}
\newcommand{\Mod}{\operatorname{mod}}

\newcommand{\rk}{\operatorname{rk}}

\newcommand{\sh}{\operatorname{sh}}

\newcommand{\sta}{\operatorname{st}}

\newcommand{\Sym}{\operatorname{Sym}}
\newcommand{\SYT}{\operatorname{SYT}}

\newcommand{\tr}{\operatorname{tr}}

\newcommand{\bul}{\bullet}
\newcommand{\dil}{\displaystyle}

\newcommand{\ds}{\displaystyle}

\newcommand{\Cat}{\operatorname{Cat}}
\newcommand{\NCM}{\operatorname{NCM}}
\newcommand{\NC}{\operatorname{NC}}
\newcommand{\CP}{\operatorname{CP}}
\newcommand{\exc}{\operatorname{exc}}
\newcommand{\Exc}{\operatorname{Exc}}
\newcommand{\stc}{\operatorname{sc}}
\newcommand{\Gal}{\operatorname{Gal}}
\newcommand{\SSYT}{\operatorname{SSYT}}
\newcommand{\csn}{\models}
\newcommand{\con}{\operatorname{ct}}
\newcommand{\conv}{\operatorname{conv}}

\newcommand{\Gaa}{\put(0,0){\circle*{3}}}
\newcommand{\Gba}{\put(10,0){\circle*{3}}}
\newcommand{\Gca}{\put(20,0){\circle*{3}}}
\newcommand{\Gda}{\put(30,0){\circle*{3}}}
\newcommand{\Gea}{\put(40,0){\circle*{3}}}
\newcommand{\Gfa}{\put(50,0){\circle*{3}}}
\newcommand{\Gga}{\put(60,0){\circle*{3}}}

\newcommand{\Gab}{\put(0,10){\circle*{3}}}
\newcommand{\Gbb}{\put(10,10){\circle*{3}}}
\newcommand{\Gcb}{\put(20,10){\circle*{3}}}
\newcommand{\Gdb}{\put(30,10){\circle*{3}}}
\newcommand{\Geb}{\put(40,10){\circle*{3}}}
\newcommand{\Gfb}{\put(50,10){\circle*{3}}}
\newcommand{\Ggb}{\put(60,10){\circle*{3}}}

\newcommand{\Gac}{\put(0,20){\circle*{3}}}
\newcommand{\Gbc}{\put(10,20){\circle*{3}}}
\newcommand{\Gcc}{\put(20,20){\circle*{3}}}
\newcommand{\Gdc}{\put(30,20){\circle*{3}}}
\newcommand{\Gec}{\put(40,20){\circle*{3}}}
\newcommand{\Gfc}{\put(50,20){\circle*{3}}}
\newcommand{\Ggc}{\put(60,20){\circle*{3}}}

\newcommand{\Gad}{\put(0,30){\circle*{3}}}
\newcommand{\Gbd}{\put(10,30){\circle*{3}}}
\newcommand{\Gcd}{\put(20,30){\circle*{3}}}
\newcommand{\Gdd}{\put(30,30){\circle*{3}}}
\newcommand{\Ged}{\put(40,30){\circle*{3}}}
\newcommand{\Gfd}{\put(50,30){\circle*{3}}}
\newcommand{\Ggd}{\put(60,30){\circle*{3}}}

\newcommand{\Gae}{\put(0,40){\circle*{3}}}
\newcommand{\Gbe}{\put(10,40){\circle*{3}}}
\newcommand{\Gce}{\put(20,40){\circle*{3}}}
\newcommand{\Gde}{\put(30,40){\circle*{3}}}
\newcommand{\Gee}{\put(40,40){\circle*{3}}}
\newcommand{\Gfe}{\put(50,40){\circle*{3}}}
\newcommand{\Gge}{\put(60,40){\circle*{3}}}

\newcommand{\Gaf}{\put(0,50){\circle*{3}}}
\newcommand{\Gbf}{\put(10,50){\circle*{3}}}
\newcommand{\Gcf}{\put(20,50){\circle*{3}}}
\newcommand{\Gdf}{\put(30,50){\circle*{3}}}
\newcommand{\Gef}{\put(40,50){\circle*{3}}}
\newcommand{\Gff}{\put(50,50){\circle*{3}}}
\newcommand{\Ggf}{\put(60,50){\circle*{3}}}

\newcommand{\Gag}{\put(0,60){\circle*{3}}}
\newcommand{\Gbg}{\put(10,60){\circle*{3}}}
\newcommand{\Gcg}{\put(20,60){\circle*{3}}}
\newcommand{\Gdg}{\put(30,60){\circle*{3}}}
\newcommand{\Geg}{\put(40,60){\circle*{3}}}
\newcommand{\Gfg}{\put(50,60){\circle*{3}}}
\newcommand{\Ggg}{\put(60,60){\circle*{3}}}

\newcommand{\GaaV}[1]{\put(0,0){\circle*{#1}}}
\newcommand{\GbaV}[1]{\put(10,0){\circle*{#1}}}
\newcommand{\GcaV}[1]{\put(20,0){\circle*{#1}}}
\newcommand{\GdaV}[1]{\put(30,0){\circle*{#1}}}
\newcommand{\GeaV}[1]{\put(40,0){\circle*{#1}}}
\newcommand{\GfaV}[1]{\put(50,0){\circle*{#1}}}
\newcommand{\GgaV}[1]{\put(60,0){\circle*{#1}}}

\newcommand{\GabV}[1]{\put(0,10){\circle*{#1}}}
\newcommand{\GbbV}[1]{\put(10,10){\circle*{#1}}}
\newcommand{\GcbV}[1]{\put(20,10){\circle*{#1}}}
\newcommand{\GdbV}[1]{\put(30,10){\circle*{#1}}}
\newcommand{\GebV}[1]{\put(40,10){\circle*{#1}}}
\newcommand{\GfbV}[1]{\put(50,10){\circle*{#1}}}
\newcommand{\GgbV}[1]{\put(60,10){\circle*{#1}}}

\newcommand{\GacV}[1]{\put(0,20){\circle*{#1}}}
\newcommand{\GbcV}[1]{\put(10,20){\circle*{#1}}}
\newcommand{\GccV}[1]{\put(20,20){\circle*{#1}}}
\newcommand{\GdcV}[1]{\put(30,20){\circle*{#1}}}
\newcommand{\GecV}[1]{\put(40,20){\circle*{#1}}}
\newcommand{\GfcV}[1]{\put(50,20){\circle*{#1}}}
\newcommand{\GgcV}[1]{\put(60,20){\circle*{#1}}}

\newcommand{\GadV}[1]{\put(0,30){\circle*{#1}}}
\newcommand{\GbdV}[1]{\put(10,30){\circle*{#1}}}
\newcommand{\GcdV}[1]{\put(20,30){\circle*{#1}}}
\newcommand{\GddV}[1]{\put(30,30){\circle*{#1}}}
\newcommand{\GedV}[1]{\put(40,30){\circle*{#1}}}
\newcommand{\GfdV}[1]{\put(50,30){\circle*{#1}}}
\newcommand{\GgdV}[1]{\put(60,30){\circle*{#1}}}

\newcommand{\GaeV}[1]{\put(0,40){\circle*{#1}}}
\newcommand{\GbeV}[1]{\put(10,40){\circle*{#1}}}
\newcommand{\GceV}[1]{\put(20,40){\circle*{#1}}}
\newcommand{\GdeV}[1]{\put(30,40){\circle*{#1}}}
\newcommand{\GeeV}[1]{\put(40,40){\circle*{#1}}}
\newcommand{\GfeV}[1]{\put(50,40){\circle*{#1}}}
\newcommand{\GgeV}[1]{\put(60,40){\circle*{#1}}}

\newcommand{\GafV}[1]{\put(0,50){\circle*{#1}}}
\newcommand{\GbfV}[1]{\put(10,50){\circle*{#1}}}
\newcommand{\GcfV}[1]{\put(20,50){\circle*{#1}}}
\newcommand{\GdfV}[1]{\put(30,50){\circle*{#1}}}
\newcommand{\GefV}[1]{\put(40,50){\circle*{#1}}}
\newcommand{\GffV}[1]{\put(50,50){\circle*{#1}}}
\newcommand{\GgfV}[1]{\put(60,50){\circle*{#1}}}

\newcommand{\GagV}[1]{\put(0,60){\circle*{#1}}}
\newcommand{\GbgV}[1]{\put(10,60){\circle*{#1}}}
\newcommand{\GcgV}[1]{\put(20,60){\circle*{#1}}}
\newcommand{\GdgV}[1]{\put(30,60){\circle*{#1}}}
\newcommand{\GegV}[1]{\put(40,60){\circle*{#1}}}
\newcommand{\GfgV}[1]{\put(50,60){\circle*{#1}}}
\newcommand{\GggV}[1]{\put(60,60){\circle*{#1}}}

\newcommand{\Gaaba}{\put(0,0){\line(1,0){10}}}

\newcommand{\Gbaca}{\put(10,0){\line(1,0){10}}}

\newcommand{\Gcada}{\put(20,0){\line(1,0){10}}}

\newcommand{\Gcacb}{\put(20,0){\line(0,1){10}}}

\newcommand{\Gdaea}{\put(30,0){\line(1,0){10}}}

\newcommand{\Geafa}{\put(40,0){\line(1,0){10}}}

\newcommand{\Gfaga}{\put(50,0){\line(1,0){10}}}

\newcommand{\Gfagb}{\put(50,0){\line(1,1){10}}}

\bcctitle{The cyclic sieving phenomenon: a survey}
\bccshorttitle{Cyclic sieving phenomenon}

\bccname{Bruce E. Sagan}
\bccshortname{Bruce Sagan}

\bccaddressa{Department of Mathematics\\ Michigan State University\\  
East Lansing, MI 48824-1027, USA}
\bccemaila{sagan@math.msu.edu}
\begin{document}
\makebcctitle

\begin{abstract}
%%%abstract text%%%
The cyclic sieving phenomenon was defined by Reiner, Stanton, and White in a
2004 paper.  Let $X$ be a finite set, $C$ be a finite cyclic group acting on
$X$, and $f(q)$ be a polynomial in $q$ with nonnegative integer coefficients.
Then the triple $(X,C,f(q))$ exhibits the {\it cyclic sieving phenomenon\/}
if, for all $g\in C$, we have
$$
\# X^g = f(\om)
$$
where $\#$ denotes cardinality, $X^g$ is the fixed point set of $g$, and $\om$
is a root of unity chosen to have the same order as $g$.   It might seem
improbable that substituting a root of unity into a polynomial with integer
coefficients would have an enumerative meaning.  But many instances of the cyclic
sieving phenomenon have now been found.
Furthermore, the proofs that this phenomenon hold often involve interesting
and sometimes deep results from representation theory.
We will survey the current literature on cyclic sieving, providing the
necessary background about representations, Coxeter groups, and other
algebraic aspects as needed.
\end{abstract} 
\thankyou{The author would like to thank Vic Reiner, Martin Rubey, Bruce
  Westbury, and an anonymous referee for 
  helpful suggestions.  And he would particularly like to thank Vic Reiner for
  enthusiastic encouragement.  This work was partially done while the author
  was a Program Officer at NSF.  The views expressed are not necessarily those
  of the NSF.} 

%%%body of article - use some of the sectioning commands etc as follows%%%
%% Section commands
%\subsection{b}
%\subsubsection{c}

%% Theorem type environments
%\begin{example}
%\end{example}
%Similarly:
%Def  - Definition
%cond - Conditions
%rem - Remark
%prop - Proposition
%theorem - Theorem
%cor - Corollary
%lemma - Lemma
%blank - a new paragraph sort of thing
%diag - Diagram
%claim - Claim
%qn - Question
%conj - Conjecture
%result - Result
%unnumbered environments
%notn - Notation
%aside - Comment
%proof  -  Proof (also gives the end of proof square)

%%% emphasis for defined words
%\defword{ }

%%% some maths commands are already defined including bold font
% (not blackboard) for the reals; integers; natural nos.
%\R
%\Z
%\N
%% some pre-defined commands for sets can be found in the file

\section{What is the cyclic sieving phenomenon?}

Reiner, Stanton, and White introduced the cyclic sieving phenomenon in their seminal 2004 paper~\cite{rsw:csp}.  In order to define this concept, we need three ingredients.  The first of these is a finite set, $X$.  The second is a finite cyclic group, $C$, which acts on $X$.  Given a group element $g\in C$, we denote its fixed point set by
\beq
\label{X^g}
X^g=\{y\in X\ :\ gy=y\}.
\eeq
We also let $o(g)$ stand for the order of $g$ in the group $C$.  One group which will be especially important in what follows will be the group, $\Om$, of roots of unity.  We let $\om_d$ stand for a primitive $d$th root of unity. The reader should think of $g\in C$ and $\om_{o(g)}\in\Om$ as being linked because they have the same order in their respective groups.  The final ingredient is a polynomial $f(q)\in\bbN[q]$, the set of polynomials in the variable $q$ with nonnegative integer coefficients.  Usually $f(q)$ will be a generating function associated with $X$.
\begin{Def}
The triple $(X,C,f(q))$ exhibits the {\em cyclic sieving phenomenon\/} (abbreviated CSP) if, for all $g\in C$, we have
\beq
\label{csp}
\# X^g = f(\om_{o(g)})
\eeq
where the hash symbol denotes cardinality.
\end{Def}

Several comments about this definition are in order.  At first blush it may
seem very strange that plugging a complex number into a polynomial with
nonnegative integer coefficients would yield a nonnegative integer, much less
that the result would count something.  However, the growing literature on the
CSP shows that this phenomenon is quite wide spread.  Of course, using linear
algebra it is always possible to find some polynomial which will satisfy the
system of equations given by~\ree{csp}.  
And in Section~\ref{rtb} we will see, via equation~\ree{chi2}, that the
polynomial can be taken to have nonnegative integer coefficients.
But the point is that $f(q)$ should
be a polynomial naturally associated with the set $X$.  In fact, letting $g=e$
(the identity element of $C$) in~\ree{csp}, it follows that  
\beq
\label{f(1)}
f(1)=\# X.
\eeq

In the case $\# C=2$, the CSP reduces to Stembridge's ``$q=-1$ phenomenon''
\cite{ste:mrp,ste:shr,ste:cbs}.  Since the Reiner-Stanton-White paper,
interest in cyclic sieving has been steadily increasing.  But since the area
is relatively new, this survey will be able to at least touch on all of the
current literature on the subject.   Cylic sieving illustrates a beautiful
interplay between cominbatorics and algebra.  We will provide all the algebraic
background required to understand the various instances of the CSP discussed.

The rest of the paper is organized as follows.  Most proofs of cyclic sieving phenomena fall into two broad categories: those involving explicit evaluation of both sides of~\ree{csp}, and those using representation theory.  In the next section, we will introduce our first example of the CSP and give a demonstration of the former type.  Section~\ref{rtb} will provide the necessary representation theory background to present a proof of the second type for the same example. 
In the following section, we will develop a paradigm for representation theory
proofs in general.  Since much of the work that has been done on CSP involves
Coxeter groups and permutation statistics, Section~\ref{cgp} will  provide a
brief introduction to them.  In Section~\ref{crg} we discuss the regular
elements of Springer~\cite{spr:ref} which have become a useful tool in proving CSP results.
Section~\ref{prs} is  concerned with Rhoades' startling theorem~\cite{rho:csp}
connecting the CSP and Sch\"utzenberger's promotion operator~\cite{sch:pme} on
rectangular standard Young tableaux.  The section following that discusses
work related to Rhoades' result.  In Section~\ref{mgm} we consider
generalizations of the CSP using more than one group or more than one
statistic.  Instances of the CSP related to Catalan numbers are discussed in
Section~\ref{ncsp}. 
The penultimate section contains some results which do not fit nicely into one
of the previous sections.  And the final section consists of various remarks. 

\bigskip

\section{An example and a proof}
\label{eap}

We will now consider a simple example of the CSP and show that~\ree{csp} holds by explicitly evaluating both sides of the equation.  For $n\in\bbN$, let $[n]=\{1,2,\ldots, n\}$.  A {\it multiset on $[n]$\/} is an unordered family, $M$, of elements of $[n]$ where repetition is allowed.  Since order does not matter, we will always list the elements of $M$ in weakly increasing order: $M=i_1 i_2\ldots i_k$ with $1\le i_1\le i_2\le \ldots\le i_k\le n$.  The set for our CSP will be
\beq
\label{(())}
X=\left(\hs{-3pt}{[n]\choose k}\hs{-3pt}\right) \stackrel{\rm def}{=} 
\{M\ :\ \mbox{$M$ is a multiset on $[n]$ with $k$ elements.}\}.
\eeq
To illustrate, if $n=3$ and $k=2$ then $X=\{11,22,33,12,13,23\}$.

For our group, we take one generated by an $n$-cycle
$$
C_n=\spn{(1,2,\ldots,n)}.
$$
Then $g\in C_n$ acts on $M=i_1 i_2\ldots i_k$ by
\beq
\label{gM}
gM=g(i_1)g(i_2)\ldots g(i_k)
\eeq
where we rearrange the right-hand side to be in increasing order. Returning to the $n=3,k=2$ case we have $C_3=\{e,(1,2,3),(1,3,2)\}$.  The action of $g=(1,2,3)$ is
\beq
\label{(1,2,3)}
\barr{lll}
(1,2,3)11=22, & (1,2,3)22=33, &(1,2,3)33=11,\\
(1,2,3)12=23, & (1,2,3)13=12, &(1,2,3)23=13.  
\earr
\eeq

To define the polynomial we will use, consider the geometric series
\beq
\label{[n]}
[n]_q=1+q+q^2+\cdots+q^{n-1}.
\eeq
This is known as a {\it $q$-analogue of $n$\/} because setting $q=1$ gives $[n]_1=n$.
Do not confuse $[n]_q$ with the set $[n]$ which has no subscript.  Now, for 
$0\le k\le n$,  define the {\it Gaussian polynomials\/} or {\it $q$-binomial coefficients\/} by
\beq
\label{gaunk}
\gau{n}{k}_q=\frac{[n]_q!}{[k]_q![n-k]_q!}
\eeq
where $[n]_q!=[1]_q[2]_q\cdots[n]_q$.  It is not clear from this definition that these rational functions are actually polynomials with nonnegative integer coefficients, but this is not hard to prove by induction on $n$.  It is well known that 
$\#(\hs{-2pt}{[n]\choose k}\hs{-2pt})={n+k-1\choose k}$.  So, in view
of~\ree{f(1)}, a natural choice for our CSP polynomial is 
$$
f(q)=\gau{n+k-1}{k}_q.
$$

We are now in a position to state our first cyclic sieving result.  It is a
special case of Theorem 1.1(a) in the Reiner-Stanton-White paper~\cite{rsw:csp}. 
\begin{theorem}
\label{multiset}
The cyclic sieving phenomenon is exhibited by the triple
$$
\left(\ \left(\hs{-3pt}{[n]\choose k}\hs{-3pt}\right),\ \spn{(1,2,\ldots,n)},\ \gau{n+k-1}{k}_q\ \right).
$$
\end{theorem}

Before proving this theorem, let us consider the case $n=3,k=2$ in detail.  First of all note that
$$
f(q)=\gau{3+2-1}{2}_q=\gau{4}{2}_q=\frac{[4]_q!}{[2]_q![2]_q!}=\frac{[4]_q[3]_q}{[2]_q}=
1+q+2q^2+q^3+q^4.
$$
Now we can verify that $\# X^g=f(\om_{o(g)})$ case by case.  If $g=e$ then $o(g)=1$, so
let $\om=1$ and compute
$$
f(\om)=f(1)=6=\# X=\# X^e.
$$
If $g=(1,2,3)$ or $(1,3,2)$ then we can use $\om=\exp(2\pi i/3)$ to obtain
$$
f(\om)=1+\om+2\om^2+\om^3+\om^4 = 2 + 2\om + 2\om^2 = 0 =\# X^g
$$
as can be seen from the table~\ree{(1,2,3)} for the action of $(1,2,3)$ which has no fixed points.

In order to give a proof by explicit evaluation, it will be useful to have some more notation.  Another way of expressing a multiset on $[n]$ is $M=\{1^{m_1},2^{m_2},\ldots,n^{m_n}\}$ where $m_i$ is the multiplicity of $i$ in $M$.
Exponents equal to one are omitted as are elements of exponent zero.  For example, $M=222355=\{2^3, 3, 5^2\}$.  Define the {\it disjoint union\/} of multisets $L=\{1^{l_1},2^{l_2},\ldots,n^{l_n}\}$ and $M=\{1^{m_1},2^{m_2},\ldots,n^{m_n}\}$ to be
$$
L\sqcup M = \{1^{l_1+m_1},2^{l_2+m_2},\ldots,n^{l_n+m_n}\}.
$$

Let $\fS_n$ denote the symmetric group of permutations of $[n]$.  Note
that~\ree{gM} defines an action of any $g\in\fS_n$ on multisets and need not
be restricted to elements of the cyclic group.  We will also want to apply the
disjoint union operation $\sqcup$ to cycles in $\fS_n$, in which case we consider each cycle as a set (which is just a multiset with all multiplicities $0$ or $1$).  To illustrate, $(1,4,3)\sqcup(1,4,3)\sqcup(3,4,5)=\{1^2,3^3,4^3,5\}$.  To evaluate 
$\# X^g$ we will need the following lemma.
\begin{lemma}
\label{gM=M}
Let $g\in\fS_n$ have disjoint cycle decomposition $g=c_1 c_2\cdots c_t$.  Then 
$gM=M$ if and only if $M$ can be written as 
$$
M=c_{r_1}\sqcup c_{r_2}\sqcup\cdots\sqcup c_{r_s}
$$
where the cycles in the disjoint union need not be distinct.
\end{lemma}
\begin{proof}
For the reverse direction, note that if $c$ is a cycle of $g$ and $\cb$ is the
corresponding set then $g\cb=\cb$ since $g$ merely permutes the elements of
$\cb$ amongst themselves. So $g$ will also fix disjoint unions of such cycles as
desired. 

To see that these are the only fixed points, suppose that $M$ is not such a disjoint union.  Then there must be some cycle $c$ of $g$ and $i,j\in[n]$ such that $c(i)=j$ but $i$ and $j$ have different multiplicities in $M$.  It follows that $gM \neq M$ which completes the forward direction.
\end{proof}

As an example of this lemma, if $g=(1,2,4)(3,5)$ then the multisets fixed by
$g$ of cardinality at most $5$ are $\{3,5\}$, $\{1,2,4\}$, $\{3^2,5^2\}$, and
$\{1,2,3,4,5\}$. 
We now apply the previous lemma to our case of interest when $C_n=\spn{(1,2,\ldots,n)}$.
In the next result, we use the standard notation $d|k$ to signify that the integer $d$ divides evenly into the integer $k$.
\bco
\label{X^g:cor}
If $X=(\hs{-2pt}{[n]\choose k}\hs{-2pt})$ and $g\in C_n$ has $o(g)=d$, then
$$
\# X^g = 
\case{\dil {n/d+k/d-1\choose k/d}}
{if $d|k$,}
{0\rule{0pt}{20pt}}
{otherwise.}
$$
\eco
\begin{proof}
Since $g$ is a power of $(1,2,\ldots,n)$ and $o(g)=d$, we have that $g$'s disjoint cycle decomposition must consist of $n/d$ cycles each of length $d$.  If $d$ does not divide $k$, then no multiset $M$ of cardinality $k$ can be a disjoint union of the cycles of $g$.  So, by Lemma~\ref{gM=M}, there are no fixed points and this agrees with the ``otherwise'' case above.

If $d|k$ then, by the lemma again, the fixed points of $g$ are those multisets obtained by choosing $k/d$ of the $n/d$ cycles of $g$ with repetition allowed.  The number of ways of doing this is the binomial coefficient in the ``if'' case.
\end{proof}

To evaluate $f(\om_{o(g)})$, we need another lemma.
\begin{lemma}
Suppose $m,n\in\bbN$ satisfy $m\Cong n\ (\Mod d)$, and let $\om=\om_d$.  Then
$$
\lim_{q\ra\om}\frac{[m]_q}{[n]_q}=
\case{\dil \frac{m}{n}}
{if $n\Cong 0\ (\Mod d)$,}
{1\rule{0pt}{20pt}}
{otherwise.}
$$
\end{lemma}
\begin{proof}
Let $m\Cong n\Cong r\ (\Mod d)$ where $0\le r<d$.  Since $1+\om+\om^2+\cdots+\om^{d-1}=0$, cancellation in~\ree{[n]} yields
$$
[m]_\om=1+\om+\om^2+\cdots+\om^{r-1}=[n]_\om.
$$
So if $r\neq0$ then $[n]_\om\neq0$ and $[m]_\om/[n]_\om=1$, proving the ``otherwise'' case.

If $r=0$ then we can write $n=\ell d$ and $m=kd$ for certain nonnegative integers $k,\ell$. It follows that
$$
\frac{[m]_q}{[n]_q}
=\frac{(1+q+q^2+\cdots+q^{d-1})(1+q^d+q^{2d}+\cdots+q^{(k-1)d})}
{(1+q+q^2+\cdots+q^{d-1})(1+q^d+q^{2d}+\cdots+q^{(\ell-1)d})}.
$$
Canceling the $1+q+q^2+\cdots q^{d-1}$ factors and plugging in $\om$ gives
$$
\lim_{q\ra\om}\frac{[m]_q}{[n]_q}=\frac{k}{\ell}=\frac{m}{n}
$$
as desired.
\end{proof}

To motivate the hypothesis of the next result, note that if $o(g)=d$ and 
$g\in C_n$ then $d|n$ by Lagrange's Theorem.
\bco
\label{gau}
If $\om=\om_d$ and $d|n$, then
$$
\gau{n+k-1}{k}_\om=
\case{\dil {n/d+k/d-1\choose k/d}}
{if $d|k$,}
{0\rule{0pt}{20pt}}
{otherwise.}
$$
\eco
\begin{proof}
In the equality above, consider the numerator and denominator of the left-hand side after canceling factorials. Since $d|n$, the product $[n]_\om[n+1]_\om\cdots[n+k-1]_\om$ starts with a zero factor and has every $d$th factor after that also equal to zero, with all the other factors being nonzero.  The product $[1]_\om[2]_\om\cdots[k]_\om$ is also of period $d$, but starting with $d-1$ nonzero factors.  It follows that the number of zero factors in the numerator is always greater than or equal to the number of zero factors in the denominator with equality if and only if $d|k$.  This implies the ``otherwise'' case.

If $d|k$, then using the previous lemma
\bea
\dil \gau{n+k-1}{k}_\om
&=&\lim_{q\ra\om}\left(\frac{[n]_q}{[k]_q} \cdot \frac{[n+1]_q}{[1]_q}\cdot
\frac{[n+2]_q}{[2]_q}\cdots\frac{[n+k-1]_q}{[k-1]_q}\right)\\[5pt]
&=&\frac{n}{k}\cdot 1 \cdots 1 \cdot \frac{n+d}{d}\cdot 1\cdots 
1\cdot\frac{n+2d}{2d}\cdot 1 \cdots\\[5pt]
&=&\frac{n/d}{k/d}\cdot\frac{n/d+1}{1}\cdot\frac{n/d+2}{2}\cdots\\[5pt]
&=&\dil {n/d+k/d-1\choose k/d}
\eea
as desired.
\end{proof}

Comparing Corollary~\ref{X^g:cor} and Corollary~\ref{gau}, we immediately have a proof of Theorem~\ref{multiset}.

\section{Representation theory background and another proof}
\label{rtb}

Although the proof just given of Theorem~\ref{multiset} has the advantage of being elementary, it does not give much intuition about why the equality~\ree{csp} holds.  Proofs of such results using representation theory are more sophisticated but also provide more insight.  We begin this section by reviewing just enough about representations to provide another demonstration of Theorem~\ref{multiset}.  Readers interested in  more information about representation theory, especially as it relates to symmetric groups, can consult the texts of James~\cite{jam:rts}, James and Kerber~\cite{jk:rts}, or Sagan~\cite{sag:sym}.  

Given a set, $X$, we can create a complex vector space, $V=\bbC X$, by considering the elements of $X$ as a basis and taking formal linear combinations.  So if $X=\{s_1,s_2,\ldots,s_k\}$ then
$$
\bbC X = \{c_1\bs_1+c_2\bs_2+\cdots+c_k\bs_k\ :\ \mbox{$c_i\in\bbC$ for all $i$}\}. 
$$
Note that when an element of $X$ is being considered as a vector, it is set in boldface type.  If $G$ is a group acting on $X$, then $G$ also acts on $\bbC X$ by linear extension.  Each element $g\in G$ corresponds to an invertible linear map $[g]$.  (Although this is the same notation as for $[n]$ with $n\in\bbN$, context should make it clear which is meant.)  If $B$ is an ordered basis for $V$ then we let $[g]_B$ denote the matrix of $[g]$ in the basis $B$.  In particular, $[g]_X$ is the permutation matrix for $g$ acting on $X$.

To illustrate these concepts, if $X=\{1,2,3\}$ then
$$
\bbC X = \{c_1\1+c_2\2+c_3\3\ :\ c_1, c_2, c_3\in\bbC\}.
$$
The group $G=\spn{(1,2,3)}$ acts on $X$ and so on $\bbC X$.  For $g=(1,2,3)$ and the basis $X$ the action is
$$
(1,2,3)\1=\2,\ (1,2,3)\2=\3,\ (1,2,3)\3=\1.
$$
Putting this in matrix form gives
\beq
\label{[(1,2,3)]}
[(1,2,3)]_X=
\left[
\barr{ccc}
0&0&1\\
1&0&0\\
0&1&0
\earr
\right].
\eeq

In general, a {\it module for $G$\/} or {\it $G$-module\/} is a vector space
$V$ over $\bbC$ where $G$ acts by invertible linear transformations.  
Most of our modules will be {\it left modules\/} with $G$ acting on the left.
Consider the {\it general linear group\/} 
$\GL(V)$ of invertible linear transformations of $V$.  If $V$ is
a $G$-module then the map $\rho:G\ra\GL(V)$ given by $g\mapsto[g]$ is called a
{\it representation\/} of $G$.   Equivalently, if $V$ is a vector space, then
a representation is a group homomorphism $\rho:G\ra\GL(V)$.  If $G$ acts on a
set $X$ then the space $\bbC X$ is called the {\it permutation module\/}
corresponding to $X$. 

Given a $G$-module, $V$, the {\it character\/} of $G$ on $V$ is the function $\chi:G\ra\bbC$ given by
$$
\chi(g)=\tr[g]
$$
where $\tr$ is the trace function.  Note that $\chi$ is well defined since the
trace of a linear transform is independent of the basis in which it is
computed.  We can now make a connection with the left-hand side of~\ree{csp}.
If a group $G$ acts on a set $X$, then the character of $G$ on $\bbC X$ is
given by
\beq
\label{chi1}
\chi(g)=\tr[g]_X=\# X^g
\eeq
since $[g]_X$ is just the permutation matrix for $g$'s action.

To see how the right side of~\ree{csp} enters in this context, write $f(q)=\sum_{i=0}^l m_i q^i$ where $m_i\in\bbN$ for all $i$.  Now suppose there is another basis, $B$, for $\bbC X$ with the property that every $g\in C$ is represented by a diagonal matrix of the form
\beq
\label{[g]_B}
[g]_B=\dia(\underbrace{1,\ldots,1}_{m_0},\underbrace{\om,\ldots,\om}_{m_1},\ldots,
\underbrace{\om^l,\ldots,\om^l}_{m_l})
\eeq
where $\om=\om_{o(g)}$.  (This may seem like a very strong assumption, but in
the next section we will see that it must hold.)  Computing the character in this basis gives
\beq
\label{chi2}
\chi(g)=\tr[g]_B = \sum_{i=0}^l m_i \om^i = f(\om).
\eeq
Comparing~\ree{chi1} and~\ree{chi2} we immediately get the CSP.  So cyclic sieving can merely amount to basis change in a $C$-module.

Sometimes it is better to use a $C$-module other than $\bbC X$ to obtain the right-hand side of~\ree{csp}.  Two $G$-modules $V,W$ are {\it $G$-isomorphic\/} or {\it $G$-equivalent\/}, written $V\iso W$, if there is a linear bijection $\phi:V\ra W$ which preserves the action of $G$, i.e., for every $g\in G$ and $\bv\in V$ we have $\phi(g\bv)=g\phi(\bv)$.  The prefix ``$G$-'' can be omitted if the group is understood from context.  To obtain $f(\om)$ as a character, any module isomorphic to $\bbC X$ will do.  So we may pick a new module with extra structure which will be useful in the proof.

We are now ready to set up the tools we will need to reprove
Theorem~\ref{multiset}.  Let $V^{\otimes k}$ denote the $k$-fold tensor
product of the vector space $V$.  If  we take $V=\bbC[n]$ which has basis
$B=\{\bi\ :\ 1\le i\le n\}$, then $V^{\otimes k}$ has a basis of the form
$$
\{\bi_1\otimes\bi_2\otimes\cdots\otimes \bi_k\ :\ \mbox{$\bi_j\in B$ for $1\le  j\le k$}\}.
$$ 
In general, for any $V$, every basis $B$ of $V$ gives rise to a basis of
$V^{\otimes k}$ consisting of $k$-fold tensors of elements of $B$.

When $V=\bbC[n]$, we consider 
the space of {\it $k$-fold symmetric tensors\/}, $\Sym_k(n)$, which is the
quotient of $V^{\otimes k}$ by the subspace generated by  
\beq
\label{Sk}
\bi_1\otimes\bi_2\otimes\cdots\otimes \bi_k\ -\ \bi_{g(1)}\otimes \bi_{g(2)}\otimes\cdots\otimes \bi_{g(k)}
\eeq
for all $g\in\fS_k$ and all tensors $\bi_1\otimes\bi_2\otimes\cdots\otimes
\bi_k$.  Note that, while we are quotienting by such differences for all
tensors, it would suffice (by linearity) to just consider the differences
obtained using $k$-fold tensors from some basis.
Let $\bi_1\bi_2\cdots\bi_k$ denote the equivalence class of
$\bi_1\otimes\bi_2\otimes\cdots\otimes \bi_k$.  These classes are indexed by
the $k$-element multisets on $[n]$ and form a basis for $\Sym_k(n)$.  So, for
example, 
$$
\Sym_2(3)=\{c_1\1\1+c_2\2\2+c_3\3\3+c_4\1\2+c_5\1\3+c_6\2\3\ :\
\mbox{$c_i\in\bbC$ for all $i$}\}.
$$

The cyclic group $C_n=\spn{(1,2,\ldots,n)}$ acts on $\bbC[n]$ and so there is an induced action on $\Sym_k(n)$ given by
\beq
\label{Sn}
g(\bi_1\bi_2\cdots\bi_k)=g(\bi_1) g(\bi_2)\cdots g(\bi_k).
\eeq
Note that when defining $\Sym_k(n)$ by~\ree{Sk}, one has $\fS_k$ acting on the
subscripts to permute the places of the vectors.  In contrast, the action
in~\ree{Sn} has $\fS_n$ acting on the basis elements themselves.  
Comparing~\ree{Sn} with~\ree{gM}, we see that $\Sym_k(n)\iso\bbC
(\hs{-2pt}{[n]\choose k}\hs{-2pt})$ as $C_n$-modules.  We will work in the
former module for our proof.   

If $g\in C_n$ then let $[g]$ and $[g]'$ denote the associated linear
transformations on $\bbC[n]$ and $\Sym_k(n)$, respectively.  By way of
illustration, when $n=3$ and $k=2$ we calculated the matrix
$[(1,2,3)]_{\{\1,\2,\3\}}$ in~\ree{[(1,2,3)]}.  On the other hand, the
table~\ree{(1,2,3)} becomes 
$$
[(1,2,3)]'_{\{\1\1,\2\2,\3\3,\1\2,\1\3,\2\3\}}
=
\left[
\barr{cccccc}
0&0&1&0&0&0\\
1&0&0&0&0&0\\
0&1&0&0&0&0\\
0&0&0&0&1&0\\
0&0&0&0&0&1\\
0&0&0&1&0&0
\earr
\right].
$$
Similarly, let $\chi$ and $\chi'$ be the respective characters of $C_n$ acting on $\bbC[n]$ and $\Sym_k(n)$.

Now suppose we can find a basis $B=\{\bb_1,\bb_2,\ldots,\bb_ n\}$ for $\bbC[n]$ which diagonalizes $[g]$, say
$$
[g]_B=\dia(x_1,x_2,\ldots,x_n).
$$
Since $B$ is a basis for $\bbC[n]$,
$$
B'=\{\bb_{i_1}\bb_{i_2}\cdots\bb_{i_k}\ :\ 1\le i_1\le i_2\le\ldots\le i_k\le n\}
$$
is a basis for $\Sym_k(n)$.  This is the crucial property of the space of symmetric tensors which makes us choose to work with them rather than in the original permutation module.  Since each element of $B$ is an eigenvector for $[g]$ acting on $\bbC[n]$, the same is true for $B'$ and $[g]'$ acting on $\Sym_k(n)$.  More precisely,
$$
g(\bb_{i_1}\bb_{i_2}\cdots\bb_{i_k})=g(\bb_{i_1}) g(\bb_{i_2})\cdots g(\bb_{i_k}) 
= x_{i_1}x_{i_2}\cdots x_{i_k} \bb_{i_1}\bb_{i_2}\cdots\bb_{i_k}.
$$
 It follows that
$$
[g]'_{B'}=\dia(x_{i_1}x_{i_2}\cdots x_{i_k}\ :\  1\le i_1\le i_2\le\ldots\le i_k\le n)
$$
and
\beq
\label{chi'}
\chi'(g)=\sum_{1\le i_1\le i_2\le\ldots\le i_k\le n} x_{i_1}x_{i_2}\cdots x_{i_k}.
\eeq
To illustrate, if $n=3$ and $[g]_{\ba,\bb,\bc}=\dia(x_1,x_2,x_3)$, then in $\Sym_2(3)$ we have
$$
\barr{lll}
g(\ba\ba)=x_1^2 \ba\ba, & g(\bb\bb)=x_2^2 \bb\bb, & g(\bc\bc)=x_3^2 \bc\bc,\\
g(\ba\bb)=x_1x_2\ba\bb, & g(\ba\bc)=x_1x_3\ba\bc, & g(\bb\bc)=x_2x_3\bb\bc,
\earr
$$
so that
\beq
\label{chi'ex}
\chi'(g)=x_1^2+x_2^2+x_3^2+x_1x_2+x_1x_3+x_2x_3.
\eeq

The expression on the right-hand side of~\ree{chi'} is called a {\it complete
  homogeneous symmetric polynomial} and denoted $h_k(x_1,x_2,\ldots,x_n)$.  It
is called ``complete homogeneous'' because it is  the sum of all monomials of
degree $k$ in the $x_i$.  It is a symmetric polynomial because it is invariant
under permutation of the subscripts of the variables.  The theory of symmetric
polynomials is intimately bound up with the representations of the symmetric and
general linear groups.  Equation~\ree{chi'ex} displays $h_2(x_1,x_2,x_3)$.  To
make use of~\ree{chi'}, we need to related complete homogeneous symmetric
functions to $q$-binomial coefficients.  This is done by taking the {\it
  principal specialization\/} which sets $x_i=q^{i-1}$ for all $i$. 
\begin{lemma}
\label{hk:lem}
For $n\ge1$ and $k\ge0$ we have
\beq
\label{hk}
h_k(1,q,q^2,\ldots,q^{n-1})=\gau{n+k-1}{k}_q.
\eeq
\end{lemma}
\begin{proof}
We do a double induction on $n$  and $k$.  For $n=1$ we have
$h_k(1)=x_1^k|_{x_1=1}=1$ and  $\left[k\atop k\right]_q=1$.  For $k=0$ it is
also easy to see that both sides are $1$. 

Assume that $n\ge2$ and $k\ge1$.  By splitting the sum for $h_k(x_1,x_2,\ldots,x_n)$ into those terms which do not contain $x_n$ and those which do, we obtain the recursion
$$
h_k(x_1,x_2,\ldots,x_n)=h_k(x_1,x_2,\ldots,x_{n-1})+x_n h_{k-1}(x_1,x_2,\ldots,x_n).
$$
Specializing yields
$$
h_k(1,q,q^2,\ldots,q^{n-1})=h_k(1,q,q^2,\ldots,q^{n-2})+q^{n-1} h_{k-1}(1,q,q^2,\ldots,q^{n-1}).
$$

Using the definition of the Gaussian polynomials in terms of $q$-factorials~\ree{gaunk},
it is easy to check that
$$
\gau{n}{k}_q=\gau{n-1}{k}_q + q^{n-k}\gau{n-1}{k-1}_q.
$$
Substituting $n+k-1$ for $n$, we see that
the right-hand side of~\ree{hk}
satisfies the same recursion as the left-hand side, which completes the proof. 
\end{proof}
 
\begin{proof} \hs{-10pt} {\bf (of Theorem~\ref{multiset})} \hs{10pt} 
Recall that  $[(1,2,\ldots,n)]$ is a linear transformation from $\bbC[n]$ to
itself.  The map has characteristic polynomial $x^n-1$ with roots 
$1,\om_n,\om_n^2,\ldots,\om_n^{n-1}$.  Since these roots are distinct, there exists a diagonalizing basis $B$ of $\bbC[n]$ with
$$
[(1,2,\ldots,n)]_B=\dia(1,\om_n,\om_n^2,\ldots,\om_n^{n-1}).
$$
Now any $g\in C_n$ is of the form $g=(1,2,\ldots,n)^i$ for some $i$ and so, since we have a diagonal representation,
$$
[g]_B=\dia(1^i,\om_n^i,\om_n^{2i},\ldots,\om_n^{(n-1)i})
=\dia(1,\om,\om^2,\ldots,\om^{n-1})
$$
where $\om=\om_n^i$ is a primitive $o(g)$-th root of unity.  The discussion leading up to equation~\ree{chi'} and the previous lemma yield
$$
\chi'(g)=h_k(1,\om,\om^2,\ldots,\om^{n-1})=\gau{n+k-1}{k}_\om.
$$
As we have already noted, $\Sym_k(n)\iso\bbC(\hs{-2pt}{[n]\choose k}\hs{-2pt})$, and so by~\ree{chi1}
$$
\chi'(g)=\#\left(\hs{-3pt}{[n]\choose k}\hs{-3pt}\right)^g.
$$
Comparing the last two equations completes the proof that the CSP holds.
\end{proof}

\section{A representation theory paradigm}
\label{rtp}

We promised to show that the assumption of $[g]$ having a diagonalization of the form~\ree{[g]_B} is not a stretch.  To do that, we need to develop some more representation theory which will also lead to a paradigm for proving the CSP.

A {\it submodule\/} of a $G$-module $V$ is a subspace $W$ which is left
invariant under the action of $G$ in that $g\bw\in W$ for all $g\in G$ and
$\bw\in W$.  The zero subspace and $V$ itself are the {\it trivial
  submodules\/}.  We say that $V$ is {\it reducible\/} if it has a nontrivial
submodule and {\it irreducible\/} otherwise.  For example, the $\fS_3$-module
$\bbC[3]$ is reducible because the 1-dimensional subspace generated by the
vector $\1+\2+\3$ is a nontrivial submodule.  It turns out that the
irreducible modules are the building blocks of all other modules in our
setting.  The next result collects together three standard results from representation
theory.  They can be found along with their proofs in Proposition 1.10.1,
Theorem 1.5.3, and Corollary 1.9.4 (respectively) of~\cite{sag:sym}.
\bth
\label{G}
Let $G$ be a finite group and consider $G$-modules which are vector spaces over $\bbC$.
\ben
\item[(a)]  The number of pairwise inequivalent irreducible $G$-modules is finite and equals  the number of conjugacy classes of $G$.
\item[(b)] (Maschke's Theorem)  Every $G$-module can be written as a direct sum of irreducible $G$-modules.
\item[(c)] Two $G$-modules are equivalent if and only if they have the same character.\hqed
\een
\eth
We note that if $G$ is not finite or if the ground field for our $G$-modules
is not $\bbC$ then the analogue of this theorem may not hold. 
Also, the forward direction of (c) is trivial (we have already been using it in the last section), while the backward direction is somewhat surprising in that one can completely characterize a $G$-module through the trace alone.

Let us construct the irreducible representations of a cyclic group $C$ with $\#C=n$.  The {\it dimension\/} of a $G$-module $V$ is its usual vector space dimension.  If $\dim V=1$ then $V$ must be irreducible.  So what do the dimension one modules for $C$ look like?  Let $g$ be a generator of $C$ and let $V=\bbC\{\bv\}$ for some vector $\bv$.  Then $g\bv=c\bv$ for some scalar $c$.  Furthermore
$$
\bv=e\bv=g^n\bv=c^n\bv.
$$
So $c^n=1$ and $c$ is an $n$th root of unity.  It is easy to verify that for
each $n$th root of unity $\om$, the map $\rho(g^j)=[\om^j]$ defines a
representation of $C$ as $j$ varies over the integers.  So we have found $n$
irreducible $G$-modules of $C$, one for each $n$th root of unity:
$V^{(0)},V^{(1)},\ldots,V^{(n-1)}$.  They clearly have different characters
(the trace of a 1-dimensional matrix being itself) and so are pairwise
inequivalent.  Finally, $C$ is Abelian and so has $\#C=n$ conjugacy classes.
Thus by (a) of Theorem~\ref{G}, we have found all the irreducible representations. 

Now given any $C$-module, $V$, part (b) of Theorem~\ref{G} says we have a
module isomorphism 
$$
V\iso\bigoplus_{i=0}^{n-1} a_i V^{(i)}
$$
where $a_i V^{(i)}$ denotes a direct sum of $a_i$ copies of $V^{(i)}$.  Since each of the summands is 1-dimensional, there is a basis $B$ simultaneously diagonalizing the linear transformations $[g]$ for all $g\in C$ as in~\ree{[g]_B}.  We can also explain the multiplicities as follows.  Extend the definition of $V^{(i)}$ to any $i\in\bbN$ by letting $V^{(i)}=V^{(j)}$ if 
$i\Cong j\ (\Mod n)$.  Now given any polynomial $f(q)=\sum_{i\ge0} m_i q^i$ with nonnegative integer coefficients, define a corresponding $C$-module
\beq
\label{V_f}
V_f=\bigoplus_{i\ge0} m_i V^{(i)}.
\eeq
Reiner, Stanton, and White identified the following representation theory paradigm for proving that the CSP holds.
\bth
\label{csp:thm}
The cyclic sieving property holds for the  triple $(X,C,f(q))$ if and only if one has $\bbC X\iso V_f$ as $C$-modules.
\eth
\begin{proof}
Note that $\# X^g$ and $f(\om_{o(g)})$ are the character values of $g\in C$ in the modules $\bbC X$ and $V_f$, respectively.  So the result now follows from Theorem~\ref{G} (c). 
\end{proof}

\section{Coxeter groups and permutation statistics}
\label{cgp}

We will now present some basic definitions and results about Coxeter groups which will be needed for later use.  The interested reader can find more information in the books of Bj\"orner and Brenti~\cite{bb:ccg} or Hiller~\cite{hil:gcg}.   A {\it finite Coxeter group\/}, $W$, is a finite group having a presentation with a set of generators $S$, and relations for each pair $s,s'\in S$ of the form
$$
(ss')^{m(s,s')}=e,
$$
where the $m(s,s')$ are positive integers satisfying
$$
\barr{cc}
m(s,s')=m(s',s);\\
m(s,s')=1 \iff s=s'.
\earr
$$
One can also define infinite Coxeter groups, but we will only need the finite
case and may drop ``finite'' as being understood in what follows.  An abstract
group $W$  may have many presentations of this form, so when we talk about a
Coxeter group we usually have a specific generating set $S$ in mind which is 
tacitly understood.  If we wish to be explicit about the generating set, then
we will refer to the {\it Coxeter system $(W,S)$\/}. Note that since
$m(s,s)=1$ we have $s^2=1$ and so the elements of $S$ are involutions.  It
follows that one can rewrite $(ss')^{m(s,s')}=e$ by bringing half the factors
to the right-hand side  
$$
\underbrace{s\ s'\ s\ s'\ s\ \cdots}_{m(s,s')}=\underbrace{s'\ s\ s'\ s\ s'\ \cdots}_{m(s,s')}.
$$

Probably the most famous Coxeter group is the symmetric group, $\fS_n$.  Here we take the generating set of adjacent transpositions $S=\{s_1,s_2,\ldots,s_{n-1}\}$ where $s_i=(i,i+1)$.  The Coxeter relations take the form
$$
\barr{l}
s_i^2=1,\\
\mbox{$s_is_j=s_js_i$ for $|i-j|\ge2$,}\\
\mbox{$s_is_{i+1}s_i=s_{i+1}s_is_{i+1}$ for $1\le i\le n-1$.}
\earr
$$
The third equation is called the {\it braid  relation\/}.  (Other
authors also include the second equation and distinguish the two by using the terms
{\it long\/} and {\it short\/} braid relations).   We will often refer
back to this example to illustrate Coxeter group concepts. 

A Coxeter group, $W$, is {\it irreducible\/} if it can not be written as a nontrivial product of two other Coxeter groups.  Irreducible finite Coxeter groups were classified by Coxeter~\cite{cox:cef}.  A list of these groups is given in Table~\ref{Cox}.  The {\it rank\/} of a Coxeter group, $\rk W$, is the minimum cardinality of a generating set $S$ and the subscript in each group name gives the rank.  If $S$ has this minimum cardinality then its elements are called {\it simple\/}. 
The middle column displays the {\it Coxeter graph\/} or {\it
  Dynkin diagram\/} of the group which has the set $S$ as its vertices with an
edge labeled $m(s,s')$ between vertices $s\neq s'$.  By convention, if
$m(s,s')=2$ (i.e., $s$  and $s'$  commute) then one omits the edge, and if
$m(s,s')=3$ then the edge is displayed without a label. 
A Coxeter group can be realized as the symmetry group of a regular polytope if
and only if its graph contains only vertices of degree one and two.  So all
these groups except $D_n$, $E_6$, $E_7$, and $E_8$ have corresponding
polytopes which are listed in the last column.

\begin{table}
\setlength{\unitlength}{1.5pt}
\begin{tabular}{c|c|c}
\mbox{Group}&\hs{20pt}\mbox{Diagram}\hs{20pt}&\mbox{Polytope}\\
\hline
$A_n$
&
\bpi(70,10)(-5,-2)
\Gaa\Gba\Gca    \Gfa\Gga
\Gaaba \Gbaca \Gfaga
\multiput(20,0)(5,0){6}{\line(1,0){4}}
\epi
&
symmetric group $\fS_{n+1}$,\\
&
&
group of the simplex\\
$B_n$
&
\bpi(70,10)(-5,-2)
\Gaa\Gba   \Gea \Gfa\Gga
\Gaaba \Geafa \Gfaga
\put(55,5){\makebox(0,0){$4$}}
\multiput(10,0)(5,0){6}{\line(1,0){4}}
\epi
&
hyperoctahedral group,\\
&
&
group of the cube/octahedron\\
$D_n$
&
\bpi(70,20)(-5,-2)
\Gaa\Gba   \Gea \Gfa\Ggb
\put(60,-10){\circle*{3}}
\Gaaba \Geafa \Gfagb
\put(50,0){\line(1,-1){10}}
\multiput(10,0)(5,0){6}{\line(1,0){4}}
\epi
&
-\\
$E_6$
&
\bpi(70,20)(-5,-2)
\Gaa\Gba\Gca\Gda\Gea \Gcb
\Gaaba \Gbaca \Gcada \Gdaea \Gcacb
\epi
&
-\\
$E_7$
&
\bpi(70,20)(-5,-2)
\Gaa\Gba\Gca\Gda\Gea\Gfa \Gcb
\Gaaba \Gbaca \Gcada \Gdaea\Geafa \Gcacb
\epi
&
-\\
$E_8$
&
\bpi(70,20)(-5,-2)
\Gaa\Gba\Gca\Gda\Gea\Gfa\Gga \Gcb
\Gaaba \Gbaca \Gcada \Gdaea\Geafa\Gfaga \Gcacb
\epi
&
-\\
$F_4$
&
\bpi(70,10)(-5,-2)
\Gaa\Gba\Gca\Gda
\Gaaba \Gbaca \Gcada
\put(15,5){\makebox(0,0){$4$}}
\epi
&
group of the $24$-cell\\
$H_3$
&
\bpi(70,10)(-5,-2)
\Gaa\Gba\Gca
\Gaaba \Gbaca
\put(15,5){\makebox(0,0){$5$}}
\epi
&
group of the dodecahedron/icosahedron\\
$H_4$
&
\bpi(70,10)(-5,-2)
\Gaa\Gba\Gca\Gda
\Gaaba \Gbaca \Gcada
\put(25,5){\makebox(0,0){$5$}}
\epi
&
group of the $120$-cell/$600$-cell\\
$I_2(m)$
&
\bpi(70,10)(-5,-2)
\Gaa\Gba
\Gaaba 
\put(5,5){\makebox(0,0){$m$}}
\epi
&
group of the $m$-gon
\end{tabular}
\caption{\label{Cox}The irreducible finite Coxeter groups}
\end{table}

There is an important function on Coxeter groups which we will need to define generating functions for instances of the CSP.  
Given $w\in W$ we can write $w=s_1s_2\cdots s_k$ where the $s_i\in S$.  
Note that here $s_i$ is just an element of $S$ and
not necessarily the $i$th generator.  Such an
expression is {\it reduced\/} if $k$ is minimal among all such expressions for
$w$ and this value of $k$ is called the {\it length\/} of $w$, written
$\ell(w)=k$.  When $W$ is of type $A_{n-1}$, i.e., the symmetric group
$\fS_n$, then there is a nice combinatorial interpretation of the length
function.  Write $w$ in {\it one-line notation\/} as $w=w_1w_2\ldots w_n$
where $w_i=w(i)$ for $i\in[n]$.   The set of {\it inversions\/} of $w$ is 
$$
\Inv w=\{(i,j)\ :\ \mbox{$i<j$ and $w_i>w_j$}\}. 
$$  
So $\Inv w$ records the places in $w$ where there is a pair of out-of-order elements.  The {\it inversion number\/} of $w$ is $\inv w=\#\Inv w$.  For example, if 
\beq
\label{w}
w=w_1w_2w_3w_4w_5=31524
\eeq
then $\Inv w=\{(1,2),(1,4),(3,4),(3,5)\}$ and $\inv w =4$.  It turns out that for type $A$,
\beq
\label{ell}
\ell(w)=\inv w.
\eeq

We can now make a connection with the $q$-binomial coefficients as follows.  Given a Coxeter system $(W,S)$ and $J\subset S$, there is a corresponding {\it parabolic subgroup\/} $W_J$ which is the subgroup of $W$ generated by $J$.  It can be shown that each coset $w W_J$ has a unique representative of minimal length.  Let $W^J$ be the set of these coset representatives and set
\beq
\label{W^J}
W^J(q)=\sum_{w\in W^J} q^{\ell(w)}.
\eeq

If $W=\fS_n$ with $S=\{s_1,s_2,\ldots,s_{n-1}\}$ as before, then remove $s_k$ from $S$ to obtain $J=S\setminus\{s_k\}$ (which generates a {\it maximal parabolic subgroup\/}).  So 
$$
(\fS_n)_J\iso\fS_k\times\fS_{n-k}
$$ 
consists of all permutations permuting the sets $\{1,2,\ldots,k\}$ and
$\{k+1,k+2,\ldots,n\}$ among themselves.  Thus multiplying $w\in\fS_n$ on the
right by an element of $(\fS_n)_J$ merely permutes $\{w_1,w_2,\ldots,w_k\}$
and $\{w_{k+1},w_{k+2},\ldots,w_n\}$ among themselves.  (We compose
permutations from right to left.)   Using~\ree{ell}, we see that the set of
minimal length coset representatives is 
\beq
\label{S_n^J}
(\fS_n)^J=\{w\in\fS_n\ :\ \mbox{$w_1<w_2<\ldots<w_k$ and $w_{k+1}<w_{k+2}<\ldots<w_n$}\}.
\eeq
A straightforward double induction on $n$ and $k$, much like the one used to prove Lemma~\ref{hk:lem}, now yields the following result.
\bpr
\label{W^J(q)}
For $W=\fS_n$ and $J=S\setm\{s_k\}$ we have
$$
W^J(q)=\gau{n}{k}_q
$$
for any $0\le k\le n$ (where for $k=0$ or $n$, $J=S$).\hqed
\epr 

As has already been mentioned, the symmetry groups of regular polytopes only
yield some of the Coxeter groups.  However, there is a geometric way to get
them all.  A {\it reflection\/} in $\bbR^n$ is a linear transformation $r_H$
which fixes a hyperplane $H$ pointwise and sends a vector perpendicular $H$ to
its negative.    A {\it real reflection group\/} is a group generated by
reflections.  It turns out that the finite real reflection groups exactly
coincide with the finite Coxeter groups; see the papers of
Coxeter~\cite{cox:dgg,cox:cef}.  
(A similar result holds in the
infinite case if one relaxes the definition of a reflection.)   Definitions
for Coxeter groups are also applied to the corresponding reflection group,
e.g., a {\it simple reflection\/} is one corresponding to an element of $S$.
The text of Benson and Grove~\cite{bg:frg} gives a nice introduction to finite
reflection groups.  For example, to get $\fS_n$ one can use the reflecting
hyperplanes $H_{i,j}$ with equation $x_i=x_j$.  If $r_{i,j}$ is the
corresponding reflection then $r_{i,j}(x_1,x_2,\ldots,x_n)$ is just the point
obtained by interchanging the $i$th and $j$th coordinates and so corresponds
to the transposition $(i,j)\in\fS_n$.  
  
We end this section by discussing permutation statistics which are intimately
connected with Coxeter groups as we have seen with the statistic $\inv$.  A
{\it statistic\/} on a finite set $X$ is a function $\sta:X\ra\bbN$.  The
statistic has a corresponding {\it weight generating function\/}
$$
f^{\sta}(X) =f^{\sta}(X;q)=\sum_{y\in X} q^{\sta y}.
$$
From Table~\ref{st} we see that on $X=\fS_3$
$$
f^{\inv}(\fS_3) = \sum_{w\in\fS_3} \inv w = 1 + 2q + 2q^2 + q^3 = (1+q)(1+q+q^2)=[3]_q!
$$
In fact, this holds for any $n$ (not just $3$); see Proposition~\ref{invmaj} below.
And the reader should compare this result with Proposition~\ref{W^J(q)} which gives the
generating function for $\inv$ over another set of permutations. 

\begin{table}
$$
\barr{ccccccc}
w         & 123 & 132 & 213 & 231 & 312 & 321\\
\inv w  &   0   &   1   &   1   &   2   &   2   &   3\\
\maj w &   0   &   2   &   1   &   2   &   1   &   3\\
\des w &   0   &   1   &   1   &   1   &   1   &   2\\
\exc w &   0   &   1   &   1   &   2   &   1   &   1
\earr
$$
\caption{\label{st}  Four statistics on $\fS_3$} 
\end{table}

There are three other statistics which will be important in what follows.  The {\it descent set\/} of a permutation $w=w_1w_2\ldots w_n$ is
$$
\Des w=\{i\ :\ w_i>w_{i+1}\}.
$$ 
We let $\des w=\#\Des w$.  Using the descents, one forms the {\it major index\/}
$$
\maj w =\sum_{i\in\Des w} i.
$$
Continuing the example in~\ree{w}, $\Des w=\{1,3\}$ since $w_1>w_2$ and
$w_3>w_4$, and so $\maj w =1+3=4$.
The major index was named for Major Percy MacMahon who introduced the
concept~\cite{mac:ipd} (or see~\cite[pp.\ 508-549]{mac:cp}). 

We say that two statistics $\sta$ and $\sta'$ on $X$ are {\it
  equidistributed\/} if $f^{\sta}(X)=f^{\sta'}(X)$.  In other words, the number of
elements in $X$ with any given $\sta$ value $k$ equals the number having 
$\sta'$ value $k$.  Comparing the first two rows of Table~\ref{st}, the reader
should suspect the following result which is not hard to prove by induction on
$n$. 
\bpr
\label{invmaj}
We have
$$
\sum_{w\in\fS_n} q^{\inv w} = [n]_q! = \sum_{w\in\fS_n} q^{\maj w}.
$$
So $f^{\inv}(\fS_n)=f^{\maj}(\fS_n)$.\hqed
\epr
Any statistic on $\fS_n$ equidistributed with $\inv$ (or $\maj$) is said to be {\it Mahonian\/}, also in tribute to MacMahon.

The last permutation statistic we need comes from the set of {\it excedances\/} which is
$$
\Exc w =\{i\ :\ w(i)>i\}.
$$
One can view excedances as ``descents'' in the cycle decomposition of $w$.  As
usual, we let $\exc w =\#\Exc w$.  In our running example $\Exc w = \{1,3\}$
since $w_1=3>1$ and $w_3=5>3$.  Comparing the distributions of $\des$ and
$\exc$ in Table~\ref{st}, the reader will see a special case of the following
proposition whose proof can be found in the text of Stanley~\cite[Proposition 1.3.12]{sta:ec1}. 
\bpr
\label{desexc}
We have $f^{\des}(\fS_n)=f^{\exc}(\fS_n)$.\hqed
\epr

Any permutation statistic equidistributed with $\des$ (or $\exc$) is said to be {\it Eulerian\/}.  The polynomial in Proposition~\ref{desexc} is called the {\it Eulerian polynomial\/}, $A_n(q)$, although some authors use this term for the polynomial $qA_n(q)$.  The first few Eulerian polynomials are
\bea
A_0(q)&=&1,\\
A_1(q)&=&1,\\
A_2(q)&=&1+q,\\
A_3(q)&=&1+4q+q^2,\\
A_4(q)&=&1+11q+11q^2+q^3,\\
A_5(q)&=&1+26q+66q^2+26q^3+q^4.
\eea
The exponential generating function for these polynomials
\beq
\label{A_n(q)}
\sum_{n\ge0}
A_n(q)\frac{t^n}{n!}
=
\frac{1-q}{e^{t(q-1)}-q}
\eeq
is attributed to Euler~\cite[p.\ 39]{knu:acp3}.

\section{Complex reflection groups and Springer's regular elements}
\label{crg}

Before stating Theorem~\ref{multiset}, we noted that it is a special case of
one part of the first theorem in the Reiner-Stanton-White paper.  To state the
full result, we need another pair of definitions.   Let $g\in\fS_N$ have
$o(g)=n$.  Say that $g$ acts {\it  freely\/} on $[N]$ if all of $g$'s cycles
are of length $n$.  So in this case $n|N$.  For example, $g=(1,2)(3,4)(5,6)$
acts freely on $[6]$.   Say that $g$ acts {\it nearly freely\/} on $[N]$ if
either it acts freely, or all of its cycles are of length $n$ except one which
is a singleton.  In the latter case, $n|N-1$.  So
$g=(1,2)(3,4)(5,6)(7)$ acts nearly freely on $[7]$.
Finally, say that the cyclic group $C$ acts {\it freely\/} or {\it nearly
  freely\/} on $[N]$ if it has a generator $g$ with the corresponding
property.  The next result is Theorem 1.1 in~\cite{rsw:csp}.  In it, 
\beq
\label{()}    
{[N]\choose k}=\{S\ :\ \mbox{$S$ is a $k$-element subset of $[N]$}\}.
\eeq
\bth
\label{rsw:first}
Suppose $C$ is cyclic and acts nearly freely on $[N]$.  The following two triples
\beq
\label{multiset2}
\left(\ \left(\hs{-3pt}{[N]\choose k}\hs{-3pt}\right),\ C,\ \gau{N+k-1}{k}_q\ \right)
\eeq
and
\beq
\label{set}
\left(\ {[N]\choose k},\ C,\ \gau{N}{k}_q\ \right)
\eeq
exhibit the cyclic sieving phenomenon.\hqed
\eth

Note that Theorem~\ref{multiset} is a special case of~\ree{multiset2} since
$C=\spn{(1,2,\ldots,N)}$ acts freely, and so also nearly freely, on $[N]$. 
At this point, the reader may have (at least) two questions in her mind.  One
might be why the multiset example was chosen to explain in detail rather than
the combinatorially simpler set example.  The reason for this is that the
representation theory proof for the latter involves alternating tensors and so
one has to worry about signs which do not occur in the symmetric tensor case.
Another puzzling aspect might be why having a  nearly free action is the right
hypothesis on $C$.  To clarify this, one needs to discuss complex reflection
groups and Springer's regular elements. 

A {\it complex (pseudo)-reflection\/} is an element of $\GL_N(\bbC)$ ($=\GL(\bbC^N)$)
which fixes a unique hyperplane in $\bbC^N$ and has finite order.  Every real
reflection can be considered as a complex reflection by extending the field.
But the complex notion  is more general since a real reflection  must
have order two.  A {\it complex reflection group\/} is a group generated by
complex reflections.  As usual, we will only be interested in the finite case.
{\it Irreducible\/} complex reflection groups are defined as in the real case,
i.e., they are the ones which cannot be written as a nontrivial product of two
other complex reflection groups.
The irreducible complex reflection groups were classified by Shephard and
Todd~\cite{st:fur}.  The book of Lehrer and Taylor~\cite{lt:urg} gives a very
lucid treatment of these groups, even redoing the Shephard-Todd
classification. 

Call an element $g$ in a finite complex reflection group $W$ {\it regular\/}
if it has an eigenvector which does not lie on any of the reflecting
hyperplanes of $W$.  An eigenvalue corresponding to this eigenvector is also
called {\it regular\/}.  In type $A$ (i.e., in the case of the symmetric
group) we have the following connection between regular elements and nearly
free actions.  
\bpr
\label{near}
Let $W= A_{N-1}$.  Then $g\in W$ is regular if and only if it acts nearly freely on $[N]$.
\epr
\begin{proof}
Assume that  $o(g)=n$ and that $g$ acts nearly freely.  Suppose first that $n|N$ and 
$$
g=(1,2,\ldots,n)(n+1,n+2,\ldots,2n)\cdots.
$$
Other elements of order $n$ can be treated similarly.  Now $(1,2,\ldots,n)$ acting on $\bbC^n$ has eigenvalue $\om^{-1}=\om_n^{-1}$ with  eigenvector $[\om,\om^2,\ldots,\om^n]^t$ ($t$ denoting transpose) all of whose entries are distinct.  So 
$$
\bv=[\om,\om^2,\ldots,\om^n,2\om,2\om^2,\ldots,2\om^n,3\om,3\om^2,\ldots,3\om^n,\ldots]^t
$$ 
is an eigenvector for $g$ lying on none of the hyperplanes $x_i=x_j$.  In the case that $n|N-1$, one can insert a $0$ in $\bv$ at the coordinate of the fixed point and preserve regularity.

Now suppose $g$  does not act nearly freely.  Consider the case when $g$ has cycles of lengths $k$ and $l$ where $k,l\ge2$ and $k\neq l$.  Without loss of generality, say $k<l$ and
$$
g=(1,2,\ldots,k)(k+1,k+2,\ldots,k+l)\cdots.
$$
Suppose  $\bv=[a_1,a_2,\ldots,a_N]^t$ is a regular eigenvector for $g$.  Then $[a_1,a_2,\ldots,a_k]^t$ must either be an eigenvector for $g'=(1,2,\ldots,k)$ or be the zero vector.  But if $\bv$ lies on none of the hyperplanes then the second possibility is out because $k\ge2$.  The eigenvectors for $g'$ are $[\om_k^i,\om_k^{2i},\ldots,\om_k^{ki}]^t$ with corresponding eigenvalues $\om_k^{-i}$ for $1\le i\le k$.  Since everything we have said also applies to $g''=(k+1,k+2,\ldots,k+l)$, the eigenvalue $\om$ of $g$ must be a root of unity with order dividing $\gcd(k,l)$.  But $\gcd(k,l)\le k<l$, so the eigenvectors of $g''$ with such eigenvalues will all have repeated entries, a contradiction.  One can deal with the only remaining case (when $g$ has at least two fixed points) similarly.
\end{proof}

In addition to the previous proposition, we will need
the following general result,  It is easy to prove from the definitions and so is left to the reader.
\begin{lemma}
\label{Gmod}
Let $V$ be a $G$-module.
\ben
\item[(a)]  If $W\sbe V$ is a $G$-submodule then the quotient space $V/W$ is also $G$-module.
\item[(b)]  If $H\le G$ is a subgroup, then $V$ is also an $H$-module.\hqed 
\een
\end{lemma}
Also, we generalize the notation~\ree{X^g}: if $V$ is a $G$-module then the {\it invariants\/} of $G$ in $V$ are
$$
V^G=\{\bv\in V\ :\ \mbox{$g\bv=\bv$ for all $g\in G$}\}.
$$

Springer's Theorem relates two algebras.  To define the first, note that
a group $G$ acts on itself by left multiplication.  The corresponding
permutation module $\bbC[G]$ is called the {\it group algebra\/} and it is an
algebra, not just a vector space, because one can formally multiply linear
combinations of group elements.   The group algebra is important in part
because it contains every irreducible representation of
$G$.  In particular, the following is true.   See Proposition~1.10.1
in~\cite{sag:sym} for more details.
\bth
\label{C[G]}
Let $G$ be a finite group with irreducible modules
$V^{(1)},V^{(2)},\ldots,V^{(k)}$ and write $\bbC[G]=\oplus_i m_i V^{(i)}$.
Then, for all $i$, 
$$
m_i=\dim V^{(i)}.
$$
so every irreducible appears with multiplicity equal to its degree.  Taking dimensions, we have

\eqed{
\sum_{i=1}^k \left(\dim V^{(i)}\right)^2=\#G
}
\eth

The second algebra is defined for any  subgroup
$W\le\GL_N(\bbC)$.  Thinking of $x_1,x_2,\ldots,x_N$ as the coordinates of
$\bbC^N$, $W$ acts on the  algebra of polynomials $\cS=\bbC[x_1,x_2,\ldots,x_N]$
by linear transformations of the $x_i$.  For example, if $W=\fS_N$ then $W$
acts on $\cS$ by permuting the variables.  The {\it algebra of coinvariants of
  $W$\/} is the quotient 
\beq
\label{A}
A=\cS/\cS^W_+
\eeq
where $\cS^W_+$ is  the ideal generated by the invariants of $W$  in $\cS$ which are homogeneous of positive degree.   Note that by Lemma~\ref{Gmod}(a), $W$ also acts on $A$.

Now let $g$ be a regular element of $W$ of order $n$ and let $C=\spn{g}$ be
the cyclic group it generates. We also let $\om=\om_n$.  Define an action of
the product group $W\times C$ on the group algebra $\bbC[W]$ by having $W$ act
by multiplication on the left and $C$ act by multiplication on the right.
These actions commute because of the associative law in $W$, justifying the
use of the direct product.  We 
also have an action of $W\times C$ on $A$:  We already noted in the previous
paragraph how $W$ acts on $A$, and we let $C$ act by 
\beq
\label{g}
g(x_i)=\om x_i
\eeq
for $i\in[N]$. The following is a beautiful theorem of Springer~\cite{spr:ref} as reformulated by Kra{\'s}kiewicz and Weyman~\cite{kw:aca}.
\bth
\label{spr}
Let $W$ be a finite complex reflection group with coinvariant algebra $A$, and let $C\le W$ be cyclically generated by a regular element.  Then $A$ and the group algebra $\bbC[W]$ are isomorphic as $W\times C$ modules.\hqed
\eth

Although $A$ and $\bbC[W]$ are isomorphic, the former has the  advantage  that it is graded, i.e., we can write $A=\oplus_{d\ge0} A_d$ where $A_d$ are the elements in $A$ homogeneous of degree $d$.  (This is well defined because the invariant ideal we modded out by is generated by homogeneous polynomials.)  And any graded algebra over $\bbC$ has a {\it Hilbert series\/}
$$
\Hilb(A;q)=\sum_{d\ge0} \dim_\bbC A_d q^d.
$$
It is this series and the previous theorem which permitted Reiner, Stanton, and White to formulate a powerful cyclic sieving result.  In it and in the following corollary, the action of the cyclic group on left cosets is by left multiplication.
\bth
\label{reg}
Let $W$ be a finite complex reflection group with coinvariant algebra $A$, and let $C\le W$ be cyclically generated by a regular element $g$.  Take any $W'\le W$ and consider the invariant algebra $A^{W'}$.  Then cyclic sieving is exhibited by the triple
$$
\left(\ W/W',\ C,\ \Hilb(A^{W'};q)\ \right).
$$
\eth
\begin{proof}
By Theorem~\ref{spr} we have an isomorphism $\phi:A\ra\bbC[W]$ of $W\times C$
modules.  So by Lemma~\ref{Gmod}(b) they are also isomorphic as
$C$-modules. Since the actions of $C$ and $W'$ commute, the invariant algebras
$A^{W'}$ and $\bbC[W]^{W'}$ are also $C$-modules and $\phi$ restricts to an
isomorphism between them.   

By~\ree{g}, the action of $g$ on the $d$th graded piece of $A^{W'}$ is just
multiplication by $\om^d$.  But this is exactly the same as the action on the
$d$th summand in the $C$-module $V_{\Hilb(A^{W'})}$ as defined for any
generating function $f$ by equation~\ree{V_f}.  So $A^{W'}\iso
V_{\Hilb(A^{W'})}$ as $C$-modules. 

As far as $\bbC[W]^{W'}$, consider the set of right cosets $W'\setm W$.
Note that $\sum_i c_i g_i\in \bbC[W]$ will be
$W'$-invariant if and only if the coefficients $c_i$ are constant on each
right coset.  So $\bbC[W]^{W'}$ is $C$-isomorphic to the permutation module
$\bbC(W'\setm W)$ with $C$ acting on the right.  But since $C$ is Abelian,
this is isomorphic to the module $\bbC(W/W')$ for the left cosets with $C$
acting (as usual) on the left. 

Putting the isomorphisms in last three paragraphs together and using Theorem~\ref{csp:thm} completes the proof.
\end{proof}

We can specialize this theorem to the case of Coxeter groups and their parabolic subgroups. One only needs the fact~\cite[\S IV.4]{hil:gcg} that, for the length generating function defined by~\ree{W^J},
$$
W^J(q)=\Hilb(A^{W_J};q).
$$
\bco
Let $(W,S)$ be a finite Coxeter system and let $J\sbe S$.  Let $C\le W$ be cyclically generated by a regular element $g$.  Then the triple
$$
\left(\ W/W_J,\ C,\ W^J(q)\ \right)
$$
satisfies the cyclic sieving phenomenon.\hqed
\eco

If we specialize even further to type $A$, then we obtain the CSP in~\ree{set}.
To see this, note first that $g$ being regular  is equivalent to its acting nearly freely by Proposition~\ref{near}.  For $J=S\setminus\{s_k\}$, the action on left cosets $\fS_N/(\fS_N)_J$ is isomorphic to the action on ${[n]\choose k}$ as can be seen using the description of the minimal length representatives~\ree{S_n^J}.  Finally, Proposition~\ref{W^J(q)} shows that the generating function is correct.

\section{Promotion on rectangular standard Young tableaux}
\label{prs}

Rhoades~\cite{rho:csp} proved an amazing cyclic sieving result about rectangular Young tableaux under the action of promotion.  While the theorem is combinatorially easy to state, his proof involves deep results about Kazhdan-Lusztig representations~\cite{kl:rcg} and a characterization of the dual canonical basis by Skandera~\cite{ska:dck}.  We will start by giving some background about Young tableaux.

A {\it partition of $n\in\bbN$\/} is a weakly decreasing sequence of positive integers $\la=(\la_1,\la_2,\ldots,\la_l)$ such that $\sum_i \la_i=n$.  We use the notation $\la\ptn n$ for this concept and call the $\la_i$ {\it parts\/}.  For example, the partitions of $4$ are $(4)$, $(3,1)$, $(2,2)$, $(2,1,1)$, and $(1,1,1,1)$.  We will use exponents to denote multiplicities just as with multisets.  So $(1,1,1,1)=(1^4)$, $(2,1,1)=(2,1^2)$, and so forth.  We will sometimes drop the parentheses and commas to simplify notation.  Partitions play an important role in number theory, combinatorics, and representation theory.  See the text of Andrews~\cite{and:tp} for more information. 

Associated with any partition $\la=(\la_1,\la_2,\ldots,\la_l)$ is its {\it
  Ferrers diagram\/}, also denoted $\la$, which consists of $l$ left-justified
rows of dots with 
$\la_i$ dots in row $i$.  We let $(i,j)$ stand for the position of the dot in row $i$ and column $j$.  For example, the partition $\la=(5,4,4,2)$ has diagram
\beq
\label{la}
\la=
\barr{ccccc}
\bul & \bul & \bul & \bul & \bul\\
\bul & \bul & \rule{5pt}{5pt}& \bul\\
\bul & \bul & \bul & \bul\\
\bul & \bul
\earr
\eeq
where the $(2,3)$ dot has been replaced by a square. 
Note that sometimes empty boxes are used instead of dots.  Also we are using English notation, as opposed to the French version where the parts are listed bottom to top.

If $\la\ptn n$ is a Ferrers diagram, then a {\it standard Young tableau, $T$, of shape $\la$\/} is a bijection $T:\la\ra[n]$ such that rows and columns increase.  We let $\SYT(\la)$ denote the set of such tableaux and  also
$$
\SYT_n=\bigcup_{\la\ptn n}\SYT(\la).
$$
Also define
$$
f^\la=\#\SYT(\la).
$$
To illustrate, 
$$
\SYT(3,2)=\left\{
\barr{ccc}
1&2&3\\
4&5
\earr
,\quad
\barr{ccc}
1&2&4\\
3&5
\earr
,\quad
\barr{ccc}
1&2&5\\
3&4
\earr
,\quad
\barr{ccc}
1&3&4\\
2&5
\earr
,\quad
\barr{ccc}
1&3&5\\
2&4
\earr
\right\}
$$
so $f^{(3,2)}=5$.  We let $T_{i,j}$ denote the element in position $(i,j)$ and
write $\sh T$ to denote the partition which is $T$'s shape.  In Rhoades'
theorem, the cyclic sieving set will be $X=\SYT(n^m)$, a set of standard Young
tableaux of {\it rectangular shape\/}. 

Partitions and Young tableaux are intimately connected with representations of
the symmetric and general linear groups.  Given $g\in\fS_n$, its {\it cycle
  type\/} is the partition gotten by arranging $g$'s cycle lengths in weakly
decreasing order.  For example, $g=(1,5,2)(3,7)(4,8,9)(6)$ has cycle type
$\la=(3,3,2,1)$. Since the conjugacy classes of $\fS_n$ consist of all
elements of the same cycle type, they are naturally indexed by partitions
$\la\ptn n$.  So by Theorem~\ref{G} (a), the partitions $\la$ also index the
irreducible $\fS_n$-modules, $V^\la$.  In fact (see Theorem 2.6.5 in~\cite{sag:sym})
\beq
\label{dimV^la}
\dim V^\la = f^\la.
\eeq
and there are various constructions which use Young tableaux of a given shape to build the corresponding representation.  Note that from Theorem~\ref{C[G]} we obtain
\beq
\label{f^la^2}
\sum_{\la\ptn n} \left(f^\la\right)^2 = n!
\eeq

If one ignores its representation theory provenance, equation~\ree{f^la^2} can be viewed as a purely combinatorial statement about tableaux.  So one could prove it combinatorially by finding a bijection between $\fS_n$ and pairs $(P,Q)$ of standard Young tableaux of the same shape $\la$, with $\la$ varying over all partitions of $n$.  The algorithm we will describe to do this is due to Schensted~\cite{sch:lid}.  It was also discovered by Robinson~\cite{rob:rsg} in a different form.  A {\it partial Young tableau\/} will be a filling of a shape with increasing rows and columns (but not necessarily using the numbers $1,\ldots,n$).  We first describe {\it insertion\/} of an element $x$ into a partial tableau $P$ with $x\not\in P$.
\ben
\item  Initialize with $i=1$ and $p=x$.
\item  If there is an element of row $i$ of $P$ larger than $p$, then remove the left-most such element and put $p$ in that position.  Now repeat this step with $i$ replaced by $i+1$ and $p$ replaced with the removed element.    
\item  When one reaches a row where no element of that row is greater then $p$, then $p$ is placed at the end of the row and insertion terminates with a new tableau, $I_x(P)$.
\een
The removals are called {\it bumps\/} and are defined so that at each step of the algorithm the rows and columns remain increasing.  For example, if
$$
T=\barr{cccc}
1&3&5&6\\
2&8&9\\
7
\earr
$$
then inserting $4$ gives
$$
\barr{ccccc}
1&3&5&6&\leftarrow4\\
2&8&9\\
7
\earr
,
\barr{ccccc}
1&3&4&6&\\
2&8&9& &\leftarrow5\\
7
\earr
,
\barr{ccccc}
1&3&4&6&\\
2&5&9\\
7& & & &\leftarrow8
\earr
,
\barr{cccc}
1&3&4&6\\
2&5&9\\
7&8
\earr
=I_4(T).
$$

We can now describe the map $w\mapsto(P,W)$.  Given $w=w_1w_2\ldots w_n$ in
1-line notation, we build a sequence of partial tableaux $P_0=\emp,
P_1,\ldots,P_n=P$ where $\emp$ is the empty tableau and $P_k=I_{w_k}P_{k-1}$
for all $k\ge1$.  At the same time, we construct a sequence
$Q_0=\emp,Q_1,\ldots,Q_n=Q$ where $Q_k$ is obtained from $Q_{k-1}$ by placing
$k$ in the unique new position in $P_{k+1}$.   To illustrate, if $w=31452$
then we obtain 
$$
\barr{cccccccc}
P_k:
&
\barr{c}
\emp\\
\rule{0pt}{10pt}
\earr,
&
\barr{c}
3\\
\rule{0pt}{10pt}
\earr,
&
\barr{c}
1\\
3
\earr,
&
\barr{cc}
1&4\\
3
\earr,
&
\barr{ccc}
1&4&5\\
3
\earr,
&
\barr{ccc}
1&2&5\\
3&4
\earr
&
=P,
\\[20pt]
Q_k:
&
\barr{c}
\emp\\
\rule{0pt}{10pt}
\earr,
&
\barr{c}
1\\
\rule{0pt}{10pt}
\earr,
&
\barr{c}
1\\
2
\earr,
&
\barr{cc}
1&3\\
2
\earr,
&
\barr{ccc}
1&3&4\\
2
\earr,
&
\barr{ccc}
1&3&4\\
2&5
\earr
&
=Q.
\earr
$$
This procedure is invertible.  Given $(P_k,Q_k)$ then we find the position $(i,j)$ of $k$ in $Q_k$.  We reverse the bumping process in $P_k$ starting with the element in $(i,j)$.  The element removed from the top row of $P_k$ then becomes the $k$th entry of $w$.  This map is called the {\it Robinson-Schensted correspondence\/} and denoted $w\stackrel{\rm R-S}{\mapsto}(P,Q)$.  We have proved the following result.
\bth
\label{RS}
For all $n\ge0$, the map $w\stackrel{\rm R-S}{\mapsto}(P,Q)$ is a bijection 

\vs{5pt}

\eqed{
 \fS_n\stackrel{\rm R-S}{\llra}\{(P,Q)\ :\ P,Q\in\SYT_n,\ \sh(P)=\sh(Q)\}.
}
\eth

In order to motivate the polyomial for Rhoades' CSP, we describe a wonderful formula due to Frame, Robinson, and Thrall~\cite{frt:hgs} for $f^{\la}$.    The {\it hook\/} of $(i,j)$ is the set of cells to its right in the same row or  below in the same column:
$$
H_{i,j}=\{(i,j')\in\la\ :\ j'\ge j\} \cup \{(i',j)\in\la\ :\ i'\ge i\}. 
$$
The corresponding {\it hooklength\/} is $h_{i,j}=\# H_{i,j}$.
The hook of $(2,2)$ in $\la=(5,4,4,2)$ is indicated by crosses is the following diagram
$$
\barr{ccccc}
\bul & \bul & \bul & \bul & \bul\\
\bul &\times&\times&\times\\
\bul &\times& \bul & \bul\\
\bul &\times
\earr
$$
and so $h_{2,2}=5$.  The next result is called the Frame-Robinson-Thrall Hooklength Formula.
\begin{theorem}
\label{frt}
If $\la\ptn n$ then

\eqed{
f^\la=\frac{n!}{\dil\prod_{(i,j)\in\la} h_{i,j}}.
}
\end{theorem}
To illustrate this theorem, the hooklengths for $\la=(3,2)$ are as follows
$$
h_{i,j}:\barr{ccc}
4&3&1\\
2&1
\earr
$$
so $f^{(3,2)}=5!/(4\cdot 3\cdot 2\cdot 1^2)=5$ as before.
The polynomial which will appear in Rhoades' cyclic sieving result is a $q$-analogue of the Hooklength Formula
\beq
\label{f^la(q)}
f^\la(q)=\frac{[n]_q!}{\dil\prod_{(i,j)\in\la} [h_{i,j}]_q}
\eeq
where $\la\ptn n$.

The only thing left to define is the group action and this will be done using
Sch\"utzenberger's promotion operator~\cite{sch:pme}.   Define $(i,j)$ to be a
{\it corner\/} of $\la$ if neither $(i+1,j)$ nor $(i,j+1)$ is in $\la$.  The
corners of $\la$ displayed in~\ree{la} are $(1,5)$, $(3,4)$, and $(4,2)$.
Given $T\in\SYT(\la)$ we define its {\it promotion\/}, $\bdy T$, by an
algorithm. 
\ben
\item Replace $T_{1,1}=1$ by a dot.
\item If the dot is in position $(i,j)$ then exchange it with with $T_{i+1,j}$ or $T_{i,j+1}$, whichever is smaller.  (If only one of the two elements exist, use it for the exchange.)  Iterate this step until $(i,j)$ becomes a corner.
\item Subtract $1$ from all the elements in the array, and replace the dot in corner $(i,j)$ by $n$ to obtain $\bdy T$.
\een
The exchanges in the second step are called {\it slides\/}.  Note that the slides are constructed so that at every step of the process the array has increasing rows and columns and so $\bdy T\in\SYT(\la)$.
By way of illustration, if\beq
\label{T}
T=\barr{ccc}
1&3&5\\
2&4&6\\
7
\earr
\eeq
then we get the sliding sequence
$$
\barr{ccc}
\bul&3&5\\
2   &4&6\\
7   
\earr,\quad
\barr{ccc}
2   &3&5\\
\bul&4&6\\
7   
\earr,\quad
\barr{ccc}
2&3   &5\\
4&\bul&6\\
7
\earr,\quad
\barr{ccc}
2&3&5   \\
4&6&\bul\\
7
\earr,\quad
\barr{ccc}
1&2&4\\
3&5&7\\
6
\earr
=\bdy T.
$$

It is easy to see that  the algorithm can be reversed step-by-step, and so promotion is a bijection on $\SYT(\la)$.  Thus the operator generates a group $\spn{\bdy}$ acting on standard Young tableaux of given shape.  In general, the action seems hard to describe, but things are much nicer for certain shapes. In particular, we have the following result of Haiman~\cite{hai:dea}.
\bth
\label{n^m}
If $\la=(n^m)$ then $\bdy^{mn}(T)=T$ for all $T\in\SYT(\la)$.\hqed
\eth
For example, if $\la=(3^2)$ then  $\bdy$ has cycle decomposition
\beq
\label{bdy}
\bdy=\left(\
\barr{ccc}
1&2&3\\
4&5&6
\earr,\quad
\barr{ccc}
1&2&5\\
3&4&6
\earr,\quad
\barr{ccc}
1&3&4\\
2&5&6
$$
\earr\
\right)
\left(\
\barr{ccc}
1&2&4\\
3&5&6
\earr,\quad
\barr{ccc}
1&3&5\\
2&4&6
$$
\earr\
\right)
\eeq
from which one sees that $\bdy^6$ is the identity map in agreement with the previous theorem.

We can now state one of the main results in Rhoades' paper~\cite{rho:csp}.
\bth
\label{rho}
If $\la=(n^m)$ then the triple
$$
\left(\ \SYT(\la),\ \spn{\bdy},\ f^\la(q)\ \right) 
$$
exhibits the cyclic sieving phenomenon.\hqed
\eth
Rhoades also has a corresponding theorem for promotion of semistandard Young
tableaux of shape $\la$, a generalization of standard Young tableaux where
repeated entries are allowed which will be discussed in Section~\ref{mgm}.
The polynomial used is 
the principal specialization of the Schur function $s_\la$, a symmetric
function which can be viewed either as encoding the character of the
irreducible module $V^\la$ or as the generating function for semistandard tableaux.  
Both the standard and semistandard results were originally conjectured by Dennis White.

\section{Variations on a theme}
\label{vt}

We will now mention several papers which have been inspired by Rhoades' work.
Stanley~\cite{sta:pe} asked if there were a more elementary proof of
Theorem~\ref{rho}.  A step in this direction for rectangles with 2 or 3 rows
was given by Petersen, Pylyavskyy, and Rhoades~\cite{ppr:pcs} who reformulated the theorem in
a more geometric way.  We will describe how this is done in the 2-row case in
detail. 

A {\it (complete) matching\/} is a graph, $M$, with vertex set $[2n]$ and  $n$ edges no two of which share a common vertex. 
The matching is {\it noncrossing\/} if it does not contain a pair of edges $ab$ and $cd$ with 
\beq
\label{acbd}
a<c<b<d
\eeq 
Equivalently, if the vertices are arranged in order around a circle, say
counterclockwise, then no pair of edges intersect.  The $5$ noncrossing
matchings on $[6]$ are displayed in Figure~\ref{match} below.  There is a bijection between $\SYT(n^2)$ and noncrossing matchings on $[2n]$  as follows.  

A {\it ballot sequence of length $n$\/} is a sequence $B=b_1b_2\ldots b_n$ of
positive integers such that for all prefixes $b_1b_2\ldots b_m$ and all
$i\ge1$, the number of $i$'s in the prefix is at least as great as the number
of $(i+1)$'s.  (One thinks of counting the votes in an election where at every
stage in the count, candidate $i$ is doing at least as well as candidate
$i+1$.)  If $\la\ptn n$ then there is a map from  tableaux $T\in\SYT_n$ to  ballot
sequences of length $n$: let $b_m=i$ if $m$ is in the $i$th row of $T$.  The
$T$ in~\ree{T} has corresponding ballot sequence $B=1212123$.  It is easy to
see that the row and column conditions on $T$ force $B$ to be a ballot
sequence.  And it is also a simple matter to construct the inverse of this
map, so it is a bijection. 

To make the connection with noncrossing matchings, suppose $T\in\SYT(n^2)$ and form a corresponding sequence of parentheses by replacing each $1$ in its ballot sequence by a left parenthesis and each $2$ with a right.  Now match the parentheses, and thus their corresponding elements of $T$, in the usual way: if a left parenthesis is immediately followed by a right parenthesis they are considered matched, remove any matched pairs and recursively match what is left.     The fact that the parentheses form a ballot sequence ensures that one gets a  noncrossing matching. And the fact that one starts with a tableau of shape $(n,n)$ ensures that the matching will be complete. For example,
$$
T=
\barr{cccc}
1&2&4&5\\
3&6&7&8
\earr
\mapsto
\barr{cccccccc}
1&2&3&4&5&6&7&8\\
(&(&)&(&(&)&)&)
\earr
\mapsto
M: 18, 23, 47, 56
$$
where numbers are shown above the corresponding parentheses and the matching is specified by its edges.  Again, an inverse is simple to construct so we have a bijection.

\bfi
$$
\partial = \left(\
\begin{pspicture}(-.9,0)(1,.9)
\SpecialCoor
\pscircle*(.6;0){.07}
\pscircle*(.6;60){.07}
\pscircle*(.6;120){.07}
\pscircle*(.6;180){.07}
\pscircle*(.6;240){.07}
\pscircle*(.6;300){.07}
\psline(.6;0)(.6;300)
\psline(.6;60)(.6;240)
\psline(.6;120)(.6;180)
\rput(.9;0){1}
\rput(.9;60){2}
\rput(.9;120){3}
\rput(.9;180){4}
\rput(.9;240){5}
\rput(.9;300){6}
\end{pspicture}
,\quad
\begin{pspicture}(-.9,0)(1,.9)
\SpecialCoor 
\pscircle*(.6;0){.07}
\pscircle*(.6;60){.07}
\pscircle*(.6;120){.07}
\pscircle*(.6;180){.07}
\pscircle*(.6;240){.07}
\pscircle*(.6;300){.07}
\psline(.6;0)(.6;180)
\psline(.6;60)(.6;120)
\psline(.6;240)(.6;300)
\rput(.9;0){1}
\rput(.9;60){2}
\rput(.9;120){3}
\rput(.9;180){4}
\rput(.9;240){5}
\rput(.9;300){6}
\end{pspicture}
,\quad
\begin{pspicture}(-.9,0)(1,.9)
\SpecialCoor 
\pscircle*(.6;0){.07}
\pscircle*(.6;60){.07}
\pscircle*(.6;120){.07}
\pscircle*(.6;180){.07}
\pscircle*(.6;240){.07}
\pscircle*(.6;300){.07}
\psline(.6;0)(.6;60)
\psline(.6;120)(.6;300)
\psline(.6;180)(.6;240)
\rput(.9;0){1}
\rput(.9;60){2}
\rput(.9;120){3}
\rput(.9;180){4}
\rput(.9;240){5}
\rput(.9;300){6}
\end{pspicture}\
\right)
\left(\
\begin{pspicture}(-.9,0)(1,.9)
\SpecialCoor 
\pscircle*(.6;0){.07}
\pscircle*(.6;60){.07}
\pscircle*(.6;120){.07}
\pscircle*(.6;180){.07}
\pscircle*(.6;240){.07}
\pscircle*(.6;300){.07}
\psline(.6;0)(.6;300)
\psline(.6;60)(.6;120)
\psline(.6;180)(.6;240)
\rput(.9;0){1}
\rput(.9;60){2}
\rput(.9;120){3}
\rput(.9;180){4}
\rput(.9;240){5}
\rput(.9;300){6}
\end{pspicture}
,\quad
\begin{pspicture}(-.9,0)(1,.9)
\SpecialCoor 
\pscircle*(.6;0){.07}
\pscircle*(.6;60){.07}
\pscircle*(.6;120){.07}
\pscircle*(.6;180){.07}
\pscircle*(.6;240){.07}
\pscircle*(.6;300){.07}
\psline(.6;0)(.6;60)
\psline(.6;120)(.6;180)
\psline(.6;240)(.6;300)
\rput(.9;0){1}
\rput(.9;60){2}
\rput(.9;120){3}
\rput(.9;180){4}
\rput(.9;240){5}
\rput(.9;300){6}
\end{pspicture}\
\right)
$$
\capt{\label{match} The action of $\bdy$ on matchings}
\efi

Applying this map to the tableaux displayed in~\ree{bdy} gives the matching
description of $\bdy$ in Figure~\ref{match}.
Clearly these cycles are obtained by rotating the matchings clockwise, and it is not hard to prove that this is always the case.  As mentioned in~\cite{ppr:pcs}, this interpretation of promotion was discovered, although never published, by Dennis White.  Note that this viewpoint makes it clear that $\bdy^{2n}(T)=T$, a special case of Theorem~\ref{n^m}.  Petersen, Pylyavskyy, and Rhoades use this setting and Springer's theory of regular elements to give a short proof of the following result which is equivalent to Theorem~\ref{rho} when $m=2$.
\bth
\label{ppr}
Let $\NCM(2n)$ be the set of noncrossing matchings on $2n$ vertices and let $R$ be rotation clockwise through an angle of  $\pi/n$.  Then the triple
$$
\left(\ \NCM(2n),\ \spn{R},\ f^{(n,n)}(q)\ \right) 
$$
exhibits the cyclic sieving phenomenon.\hqed
\eth
This trio of authors applies similar ideas to the 3-row case by replacing noncrossing matchings with $A_2$ webs.  Webs were introduced by Kuperberg~\cite{kup:sr2} to index a basis for a vector space used to describe the invariants of a tensor product of irreducible representations of a rank $2$ Lie algebra.

Westbury~\cite{wes:itc1} was able to generalize the Petersen-Pylyavskyy-Rhoades proof to a much wider setting.    To understand his contribution, consider the coinvariant algebra, $A$, for a Coxeter group $W$ as defined by~\ree{A}.  If $V^\la$ is an irreducible  module of $W$, then the corresponding {\it fake degree polynomial\/} is
\beq
\label{F^la(q)}
F^\la(q)=\sum_{d\ge0} m_d q^d
\eeq
where $m_d$ is the multiplicity of $V^\la$ in the $d$th graded piece of $A$.  We can extend this to any representation $V$ of $W$ by linearity.  That is, write $V=\sum_\la n_\la V^\la$ in terms of irreducibles and then let
$$
F^V(q)=\sum_\la n_\la F^{\la}(q).
$$ 
Since the coinvariant algebra is $W$-isomorphic to the regular representation, Theorem~\ref{C[G]} for $W=A_n$ and~\ree{dimV^la} imply that  
\beq
\label{F^la(1)}
F^\la(1)=f^\la
\eeq
so that the fake degree polynomial is a $q$-analogue of the number of
standard Young tableaux.  It can be obtained from the $q$-Hooklength
Formula~\ree{f^la(q)} by multiplying by an appropriate power of $q$.  

The next result, which can be found in Westbury's article~\cite{wes:itc1}, follows easily from Proposition 4.5 in Springer's original paper~\cite{spr:ref}.
\bth
\label{wes:thm}
Let $W$ be a finite complex reflection group and let $C\le W$ be cyclically generated by a regular element $g$.  Let $V$ be a $W$-module with a basis $B$ such that $gB=B$.  Then the triple
$$
\left(\ B,\ C,\ F^V(q)\ \right)
$$
exhibits the cyclic sieving phenomenon. \hqed
\eth
Petersen-Pylyavskyy-Rhoades used webs of types $A_1$ and $A_2$ for the bases $B$, and the irreducible symmetric group modules $V^{(n,n)}$ and $V^{(n,n,n)}$ to determine the fake degree polynomials.  Westbury is able to produce many interesting results by using other bases and any highest weight representation of a simple Lie algebra.  Crystal bases and Lusztig's theory of based modules~\cite[Ch.\ 27--28]{lus:iqg} come into play.

A  CSP related to Rhoades' was studied by Petersen and Serrano~\cite{ps:csl}.  Consider the
Coxeter group $B_n$ with generating set $\{s_1,s_2,\ldots,s_n\}$ where $s_1$
is the special element corresponding to the endpoint of the Dynkin diagram
adjacent to the edge labeled 4.  Every finite Coxeter group has a {\it longest
  element\/} which is  $w_0$ with $l(w_0)>l(w)$ for all $w\in W\setm\{w_0\}$.
Let $R(w_0)$ denote the set of reduced expressions for $w_0$ in $B_n$ where
$\ell(w_0)=n^2$.  We will represent each such expression by the sequence of
its subscripts.  For example, in $B_3$ the expression
$w_0=s_1s_3s_2s_3s_1s_2s_3s_1s_2$ would become the word $132312312$.    Act on
$R(w_0)$ by rotation, i.e., remove the first element of the sequence and move
it to the end.  In our example, $132312312\mapsto 323123121$. 
(One has an analogous action in any Coxeter system $(W,S)$ with $S$ simple where, if $s_i$ rotates from the front of $w_0$, then at the back  it is replaced by $s_j = w_0 s_i w_0$ which is also simple.  In type $B_n$, one has the nice property that $s_i=s_j$.) 
For the final ingredient, it is easy to see that the definitions of the permutation
statistics from Section~\ref{cgp} all make sense when applied to sequences of
integers.   

The main theorem of Petersen and Serrano~\cite{ps:csl} can now be stated as follows.  We will henceforth use $C_N$ to denote a cyclic group of cardinality $N$.
\bth
In $B_n$, let $C_{n^2}$ be the cyclic group of rotations of elements of $R(w_0)$.  Then the triple
$$
\left(\ R(w_0),\ C_{n^2},\  q^{-n{n\choose2}} f^{\maj}(R(w_0;q))\ \right)
$$ 
exhibits the cyclic sieving property. \hqed
\eth  
They prove this result by using  bijections due to Haiman~\cite{hai:mis,hai:dea} to relate promotion on tableaux of shape $(n^n)$ to rotation of words in $R(w_0)$.  In fact, they also show
$$
q^{-n{n\choose2}} f^{\maj}(R(w_0;q))=f^{(n^n)}(q)
$$  
where the latter is the $q$-analogue of the Hooklength Formula~\ree{f^la(q)}.  The tableaux version of the previous theorem also appeared in~\cite{rho:csp}, but the proof had gaps which Petersen and Serrano succeeded in filling.

Pon and Wang~\cite{pw:pes} have made steps towards finding an analogue of Rhoades' result for staircase tableaux.  The {\it staircase shape\/} is the one corresponding to the partition $\stc_n=(n,n-1,\ldots,1)$.  Note that $\stc_n\ptn{n+1\choose2}$.  The following result of Edelman and Greene~\cite{eg:bt} shows that staircase tableaux are also well behaved with respect to promotion.
\bth
\label{sc}
We have  $\bdy^{{n+1\choose2}}(T)=T^t$ for all $T\in\SYT(\stc_n)$ where $t$
denotes transpose. \hqed
\eth
Haiman~\cite{hai:dea} has a  theory of ``generalized staircases'' which considers for which partitions $\la\ptn N$, $\bdy^N(T)$ has a nice description for all $T\in\SYT(\la)$.  It includes Theorems~\ref{n^m} and~\ref{sc}.

Note that we have $\bdy^{n(n+1)}(T)=T$ for all $T\in\SYT(n^{n+1})$ as well as
for all $T\in\SYT(\stc_n)$.  So it is natural to try and relate these two sets
of tableaux and the action of promotion on them.   In particular, Pon and Wang
define an injection $\io:\SYT(\stc_n)\ra\SYT(n^{n+1})$ commuting with $\bdy$.
To construct this map, we need another operation of Sch\"utzenberger called
evacuation~\cite{sch:rcs}.  Given $T\in\SYT(\la)$ where $\la\ptn N$, we create
its {\it evacuation\/}, $\ep(T)$, by doing $N$ promotions.  After the $i$th
promotion, one puts $N-i+1$ in the ending position of the dot and this element does not
move in any further promotions.  The slide sequence for a promotion terminates
when the position $(i,j)$ of the dot is such that $(i+1,j)$ is either outside
$\la$ or contains a fixed element, and the same is true of $(i,j+1)$.  We will
illustrate this on 
\beq
\label{scT}
T=\barr{ccc}
1&3&6\\
2&4\\
5
\earr
\eeq
where fixed elements will be typeset in boldface:
$$
\barr{rcccccl}
T=\barr{ccc}
1 & 3 & 6\\
2 & 4\\
5
\earr
&
\stackrel{\bdy}{\mapsto}
&
\barr{ccc}
1 & 2 & 5\\
3 & \6\\
4
\earr
&
\stackrel{\bdy}{\mapsto}
&
\barr{ccc}
1 & 4 & \5\\
2 & \6\\
3
\earr
&
\stackrel{\bdy}{\mapsto}
&
\barr{ccc}
1 & 3 & \5\\
2 & \6\\
\4
\earr
\\[25pt]
&
\stackrel{\bdy}{\mapsto}
&
\barr{ccc}
1 & 2 & \5\\
\3& \6\\
\4
\earr
&
\stackrel{\bdy}{\mapsto}
&
\barr{ccc}
1 & \2& \5\\
\3& \6\\
\4
\earr
&
\stackrel{\bdy}{\mapsto}
&
\barr{ccc}
\1& \2& \5\\
\3& \6\\
\4
\earr
=\ep(T).
\earr
$$

Now given $T\in\SYT(\stc_n)$ we define $\io(T)$ as follows.  Construct
$\ep(T)$ and complement each entry, $x$, by replacing it with $n(n+1)+1-x$.
Then reflect the resulting tableau in the anti-diagonal.  Finally, paste  $T$
and the reflected-complemented tableau together along their staircase portions
to obtain $\io(T)$ of shape $(n^{n+1})$.  Continuing the example~\ree{scT}, we
see that the complement of $\ep(T)$ is 
$$
\barr{ccc}
12 &11 &8\\
10 &7\\
9 
\earr
$$
so that
$$
\io(T)=
\barr{ccc}
1 &3 &6 \\
2 &4 &8 \\
5 &7 &11\\
9 &10&12
\earr.
$$
As mentioned, Pon and Wang~\cite{pw:pes} prove the following result about their map $\io$.
\bth
We have
$$
\bdy(\io(T))=\io(\bdy(T)).
$$
for all $T\in\SYT(\stc_n)$.\hqed
\eth

To get a CSP for staircase tableaux, one needs an
appropriate polynomial.   The previous theorem permits one to obtain
information about the cycle structure of the action of $\bdy$ on staircase
tableaux from what is already know about rectangular tableaux.  It is hoped
that this will aid in the search for the correct polynomial.

\section{Multiple groups and multiple statistics}
\label{mgm}

It is natural to ask whether the cyclic sieving phenomenon can be extended to other groups.  Indeed, it is possible to define an analogue of the CSP for Abelian groups by considering them as products of cyclic groups.  For this we will also need to use multiple statistics, one for each cyclic group.  In this section we examine this idea, restricting to the case of two cyclic groups to illustrate the ideas involved.

Bicyclic sieving was first defined by Barcelo, Reiner, and
Stanton~\cite{brs:bd}.  Recall that $\Om$ is the (infinite) group of roots of unity.
\begin{Def}
Let $X$ be a finite set.  Let $C,C'$ be finite cyclic groups with $C\times C'$ acting on $X$, and fix embeddings $\om:C\ra \Om$, $\om':C'\ra\Om$.  Let $f(q,t)\in\bbN[q,t]$.  The triple $(X,C\times C',f(q,t))$ exhibits the {\it bicyclic sieving phenomenon\/} or {\it biCSP\/} if, for all $(g,g')\in C\times C'$, we have
\beq
\label{bicsp}
\# X^{(g,g')} = f(\om(g),\om'(g')).
\eeq
\end{Def}

Note that in the original definition of the CSP we did not insist on an
embedding of $C$ in $\Om$ but just used any root of unity with the same
order as $g$ (although Reiner, Stanton, and White did use an embedding in the
definition from
their original paper).  But this does not matter because if evaluating
$f(q)\in\bbR[q]$ at a primitive $d$th root of unity gives a real number, then
any $d$th root will  give the same value.  But in the definition of the biCSP
the choice of embeddings can make the difference  whether~\ree{bicsp}
holds or not.  To illustrate this, we use an example from a paper of Berget,
Eu, and Reiner~\cite{ber:ccs}. 

Take $X=\{1,\om,\om^2\}$ where $\om=\exp(2\pi i/3)$ and   let $C= C'=X$.   Define the action of  any $(\al,\be)\in C\times C'$ on $\ga\in X$ by
$$
(\al,\be) \ga=\al\be\ga
$$ 
where on the right multiplication is being done in $\bbC$.
Finally, consider the polynomial
$$
f(q,t)=1+ut+u^2t^2.
$$
If for the embeddings one takes identity maps, then it is easy to verify
case-by-case that $(X,C\times C',f(q,t))$ exhibits the biCSP.   But if one
modifies the embedding of $C'$ to be the one which takes $\om\ra\om^{-1}$ then
this is false.  For example, consider  $(\om,\om)$ whose action on
$X$ is the cycle $(1,\om^2,\om)$.  Using the first embedding pair we find
$f(\om,\om)=1+\om^2+\om=0$ reflecting the fact that there are no fixed points.
However, if one uses the second pair then the computation is
$f(\om,\om^{-1})=1+1+1=3$ which does not agree with the action.  

Part of the motivation for studying the biCSP was to generalize the notion of a bimahonian distribution to other Coxeter groups $W$.  It was observed by Foata and Sch\"utzenberger~\cite{fs:mii} that certain pairs of Mahonian statistics such as $(\maj(w),\inv(w))$ and $(\maj(w),\maj(w^{-1}))$ had the same joint distribution over $\fS_n$.  To define a corresponding bivariate distribution on  $W$, consider a field automorphism $\si$ lying in the Galois group
$\Gal(\bbQ[\exp(2\pi i/m)]/\bbQ)$ where $m$ is taken large enough so that the
extension $\bbQ[\exp(2\pi i/m)]$ of $\bbQ$ contains all the matrix entries of
all elements in the reflection representation of $W$.    Use the fake degree
polynomials~\ree{F^la(q)} to define the {\it $\si$-bimahonian distribution\/}
on $W$ by 
$$
F^{\si}(q,t)=\sum_{V^\la} F^{\si(\la)}(q) F^{\ol{\la}}(t)
$$
where the sum is over all irreducible $V^\la$, and $V^{\si(\la)}$ and $V^{\ol{\la}}$ are the modules defined (via Theorem~\ref{G} (c)) by the characters
$$
\mbox{$\chi^{\si(\la)}(w)=\si(\chi^\la(w))$ and $\chi^{\ol\la}(w)=\ol{\chi^\la(w)}$}
$$
for all $w\in W$.

To apply Springer's theory in this setting, suppose $g$ and $g'$ are regular elements of $W$ with regular eigenvalues $\om$ and $\om'$, respectively.  Consider the embeddings of  $C=\spn{g}$ and $C'=\spn{g'}$ into $\Om$ uniquely defined by mapping $g\mapsto \om^{-1}$ and $g'\mapsto(\om')^{-1}$.  (The reason for using inverses should be clear from the proof of Proposition~\ref{near}.)  Given $\si$ as above, pick $s\in\bbN$ so that $\si(\om)=\om^s$.  In this setting, Barcelo, Reiner, and Stanton~\cite{brs:bd} prove the following result.
\bth
\label{BRS}
 Let $W$ be  a finite complex reflection group with regular elements $g,g'$.  With the above notation, consider the $\si$-twisted action of $\spn{g}\times \spn{g'}$ on $W$ defined by
$$
(g,g')w=g^sw(g')^{-1}.
$$
Then the triple
$$
\left(\ W,\ \spn{g}\times \spn{g'},\ F^{\si}(q,t)\ \right)
$$
exhibits the bicyclic sieving phenomenon.\hqed
\eth

Reiner, Stanton and White had various conjectures about biCSPS exhibited by action on nonnegative integer matrices via row and column rotation.  These were communicated to Rhoades and proved by him in~\cite{rho:hlp}.  To talk about them, we will need some definitions.  A {\it composition of $n$ of length $l$\/} is a (not necessarily weakly decreasing) sequence $\mu=(\mu_1,\mu_2,\ldots,\mu_l)$ of nonnegative integers with $\sum_i \mu_i=n$.   We write $\mu\csn n$ and $\ell(\mu)=l$.  A {\it semistandard Young tableau of shape $\la$ and content $\mu$\/} is a function $T:\la\ra\bbP$ (where $\bbP$ is the positive integers) such that rows weakly increase, columns strictly increase, and $\mu_k$ is the number of $k$'s in $T$.  For example,
$$
T=
\barr{cccccc}
1&1&1&2&3&5\\
2&3&3&3\\
5&5
\earr
$$
is a semistandard tableau of shape $(6,4,2)$ and content $(3,2,4,0,3)$.  We write $\ctt T=\mu$ to denote the content and let 
$$
\SSYT(\la,\mu)=\{ T\ :\ \mbox{$\sh T=\la$ and $\ctt T=\mu$}\}.
$$

The {\it Kostka numbers\/} are 
$$
K_{\la,\mu}=\#\SSYT(\la,\mu).
$$
Note that if $\la\ptn n$ then $K_{\la,(1^n)}=f^\la$, the number of standard
Young tableaux.  Just like $f^\la$, the $K_{\la,\mu}$ have a nice
representation theoretic interpretation. 
Given $\mu=(\mu_1,\mu_2,\ldots,\mu_l)\csn n$ there is a corresponding {\it Young subgroup\/} of $\fS_n$ which is
$$
\fS_\mu=\fS_{\{1,2,\ldots,\mu_1\}}\times\fS_{\{\mu_1+1,\mu_1+2,\ldots,\mu_1+\mu_2\}}\times
\cdots\times \fS_{\{n-\mu_l+1,n-\mu_l+2,\ldots,n\}}
$$
where for any set $X$ we let $\fS_X$ be the group of permutations of $X$.
Consider the usual action of $\fS_n$ on left cosets $\fS_n/\fS_\mu$ and let
$M^\mu$ be the corresponding permutation module.  Decomposing into $\fS_n$-irreducibles gives 
$$
M^\mu=\sum_\la K_{\la,\mu} V^\la,
$$
 so that the Kostka numbers give the multiplicities of this decomposition.  The polynomials  which appear in these biCSPs are a $q$-analogue of the $K_{\la,\mu}$ called the {\it Kostka-Foulkes polynomials\/}, $K_{\la,\mu}(q)$.  We will not give a precise definition of them, but just say that they can be viewed in a couple of ways.  One is as the elements of transition matrices between the Schur functions and Hall-Littlewood polynomials (another basis for the algebra of symmetric functions over the field $\bbQ(q)$).  Another is as the generating function for a statistic called charge on semistandard tableaux which was introduced by Lascoux and Sch\"utzenberger~\cite{ls:mp}.

Knuth~\cite{knu:pmg} generalized the Robinson-Schensted correspondence in
Theorem~\ref{RS} to semistandard tableaux.  Given compositions $\mu,\nu\csn n$ let $\Mat_{\mu,\nu}$ denote the set of all matrices with nonnegative integer entries whose row sums are given by $\mu$ and whose column sums are given by $\nu$.   For example,
$$
\Mat_{(2,1),(1,0,2)}=
\left\{
\left[\barr{ccc}
1&0&1\\
0&0&1
\earr\right]
,
\left[\barr{ccc}
0&0&2\\
1&0&0
\earr\right]
\right\}.
$$
Note that $\Mat_{(1^n),(1^n)}$ is the set of $n\times n$ permutation matrices.

Given $M\in\Mat_{\mu,\nu}$ we first convert the matrix into a two-rowed array with the columns lexicographically ordered (the first row taking precedence) and column $ {\scriptstyle{i}\atop \scriptstyle{j}}$ occuring $M_{i,j}$ times.  To illustrate, 
$$
M=\left[
\barr{ccc}
1&2&0\\
1&0&1
\earr
\right]
\mapsto
\barr{ccccc}
1&1&1&2&2\\
1&2&2&1&3
\earr.
$$
Now use the same insertion algorithm as for the Robinson-Schensted correspondence to build a semistandard tableau $P$ from the elements of the bottom line of the array, while the elements of the top line are placed in a tableau $Q$ to maintain 
equalities of shapes,
$$
\barr{cccccccc}
P_k:
&
\barr{c}
\emp\\
\rule{0pt}{10pt}
\earr,
&
\barr{c}
1\\
\rule{0pt}{10pt}
\earr,
&
\barr{cc}
1&2\\
\rule{0pt}{10pt}
\earr,
&
\barr{ccc}
1&2&2\\
\rule{0pt}{10pt}
\earr,
&
\barr{ccc}
1&1&2\\
2
\earr,
&
\barr{cccc}
1&1&2&3\\
2
\earr
&
=P,
\\[20pt]
Q_k:
&
\barr{c}
\emp\\
\rule{0pt}{10pt}
\earr,
&
\barr{c}
1\\
\rule{0pt}{10pt}
\earr,
&
\barr{cc}
1&1\\
\rule{0pt}{10pt}
\earr,
&
\barr{ccc}
1&1&1\\
\rule{0pt}{10pt}
\earr,
&
\barr{ccc}
1&1&1\\
2
\earr,
&
\barr{cccc}
1&1&1&2\\
2
\earr
&
=Q.
\earr
$$
We denote this map by $M\stackrel{\rm R-S-K}{\mapsto} (P,Q)$.
Reversing the steps of the algorithm is much like the standard tableau case
once one realizes that during insertion equal elements enter into $Q$ from left to right.
Also, it should be clear that applying this map to permutation matrices
corresponds with the original algorithm.  So the full Robinson-Schensted-Knuth
Theorem~\cite{knu:pmg} is as follows. 
\bth
For all $\mu,\nu\csn n$, the map $M\stackrel{\rm R-S-K}{\mapsto} (P,Q)$ is a bijection

\vs{5pt}

\eqed{
\Mat_{\mu,\nu}\stackrel{\rm R-S-K}{\llra}\{(P,Q)\ :\ \con P=\nu,\ \con Q=\mu,\ \sh P = \sh Q\}.
}
\eth

To motivate Rhoades' result, note that 
by specializing Theorem~\ref{BRS} to type $A$, one can obtain the following
theorem. 
\bth
Let $C_n\times C_n$ act on $n\times n$ permutation matrices by rotation of rows in the first component and of columns in the second.  Then
$$
\left(\ \Mat_{(1^n),(1^n)},\quad C_n\times C_n,\quad \ep_n(q,t)\sum_{\la\ptn n}K_{\la,(1^n)}(q) K_{\la,(1^n)}(t)\ \right)
$$
exhibits the bicyclic sieving phenomenon, where

\vs{5pt}

\eqed{
\ep_n(q,t)
=\case{(qt)^{n/2}}{if $n$ is even,}{1}{if $n$ is odd.}
}
\eth
It is natural to ask for a generalization of this theorem to tableaux of arbitrary content, and that is one of the conjectures demonstrated by Rhoades in~\cite{rho:hlp}.   The proof uses a generalization of the  R-S-K correspondence due to Stanton and White~\cite{sw:sar}.  
\bth
Let $\mu,\nu\csn n$ have cyclic symmetries of orders $a|\ell(\mu)$ and $b|\ell(\nu)$, respectively.  Let 
$C_{\ell(\mu)/a}\times C_{\ell(\nu)/b}$ act on $\Mat_{\mu,\nu}$ by $a$-fold rotation of rows in the first component and $b$-fold rotation of columns in the second.  Then
$$
\left(\ \Mat_{\mu,\nu},\quad C_{\ell(\mu)/a}\times C_{\ell(\nu)/b},\quad 
\ep_n(q,t)\sum_{\la\ptn n}K_{\la,\mu}(q) K_{\la,\nu}(t)\ \right)
$$
exhibits the bicyclic sieving phenomenon.\hqed
\eth

Rather than considering two Mahonian statistics, one could take one Mahonian and one Eulerian.  
Given two statistics, $\sta$ and $\sta'$ on a set $X$ we let
$$
f^{\sta,\sta'}(X;q,t)=\sum_{y\in X} q^{\sta y}t^{\sta' y}.
$$
Through their study of Rees products of posets, Shareshian and Wachs~\cite{sw:eqf, sw:qep,sw:phr} were lead to consider the pair $(\maj,\exc)$.  
They proved, among other things, the following generalization of~\ree{A_n(q)}
$$
\sum_{n\ge0} f^{\maj,\exc}(\fS_n;q,t)\frac{x^n}{[n]_q!}
=\frac{(1-tq)\exp_q(x)}{\exp_q(qtx)-qt\exp_q(x)},
$$
where
$$
\exp_q(x)=\sum_{n\ge0} \frac{x^n}{[n]_q!}.
$$
(Actually, they proved a stronger result which also keeps track of the number of fixed points of $w$, but that will not concern us here.)

Sagan, Shareshian, and Wachs~\cite{ssw:eqf} used the Eulerian quasisymmetric
functions as developed in the earlier papers by the last two authors, as well
as a result of D\'esarm\'enien~\cite{des:fsa} about evaluating principal
specializations at roots of unity, to demonstrate the following cyclic sieving
result.  Note the interesting feature that one must take the difference 
$\maj - \exc$.
\bth
\label{S(la)}
Consider the action of the symmetric group on itself by conjugation and let $\fS(\la)$ denote the conjugacy class of permutations of cycle type $\la\ptn n$.  Then the triple
$$
\left(\ \fS(\la),\ \spn{(1,2,\ldots,n)},\ f^{\maj,\exc}\left(\fS(\la);q,q^{-1}\right)\ \right)
$$
exhibits the cyclic sieving phenomenon. \hqed
\eth

\section{Catalan CSPs}
\label{ncsp}

We will now consider cyclic sieving phenomena where analogues of the Catalan numbers play a role.    These will include noncrossing partitions and facets of cluster complexes.  Noncrossing matchings have already come into play in Theorem~\ref{ppr}.

A {\it set partition\/} of a finite set $X$ is a collection
$\pi=\{B_1,B_2,\ldots,B_k\}$ of nonempty subsets such that $\uplus_i B_i=X$ where $\uplus$
denotes disjoint union.  We write $\pi\ptn X$ and the $B_i$ are called {\it
  blocks\/}.  A set partition $\pi\ptn[n]$ is called {\it noncrossing\/} if
condition~\ree{acbd} never holds when $a,b$ are in one block of $\pi$ and
$c,d$ are in another.  Equivalently, with the usual circular arrangement of
$1,\ldots, n$, the convex hulls of different blocks do not intersect.  Let
$\NC(n)$ denote the set of noncrossing partitions of $[n]$.  Noncrossing
partitions were introduced by Kreweras~\cite{kre:pnc} and much information
about them can be found in the survey article of Simion~\cite{sim:csn} and the
memoir of Armstrong~\cite{arm:gnp}. 

The noncrossing partitions are enumerated by the {\it Catalan numbers\/}
$$
\# \NC(n)=\Cat_n\stackrel{\rm def}{=} \frac{1}{n+1}{2n\choose n}.
$$
(Often in the literature $C_n$ is used to denote the $n$th Catalan number, but this would conflict with our notation for cyclic groups.)
These numbers have already been behind the scenes as we also have $\# \SYT(n,n)=\#\NCM(2n)=\Cat_n$.  There are a plethora of combinatorial objects enumerated by $\Cat_n$, and Stanley maintains a list~\cite{sta:ca} which the reader can consult for more examples.  

A $q$-analogue of  $\Cat_n$, 
$$
\Cat_n(q)=\frac{1}{[n+1]_q}\gau{2n}{n}_q,
$$
was introduced by F\"urlinger and Hofbauer~\cite{fh:qcn}.  
The following result follows from Theorem~7.2 in Reiner, Stanton, and White's original paper~\cite{rsw:csp} where they proved a more refined version which also keeps track of the number of blocks.
\bth
Let 
\hs{-.5pt}
$C_n$ 
\hs{-.5pt}
act 
\hs{-.5pt}
on 
\hs{-.5pt}
$\NC(n)$
\hs{-.5pt}
by
\hs{-.5pt}
rotation.  Equivalently, let
$g\in\spn{(1,2,\ldots,n)}$ act on  
$\pi=\{B_1,B_2,\ldots,B_k\}\in\NC(n)$ by $g\pi=\{gB_1,gB_2,\ldots,gB_n\}$ where $gB_i$ is defined by~\ree{gM}.  Then the triple
$$
\left(\ \NC(n),\ C_n,\ \Cat_n(q)\ \right)
$$
exhibits the cyclic sieving phenomenon.\hqed
\eth

In~\cite{br:csn}, Bessis and Reiner generalize this theorem to certain complex reflection groups.  First consider a finite Coxeter group $W$ of rank $n$.  It can be shown that we always have a factorization
$$
f^\ell(W;q)=\sum_{w\in W} q^{\ell(w)} = \prod_{i=1}^n [d_i]_q
$$
where the positive integers $d_1\le d_2\le \ldots\le d_n$ are called the {\it degrees\/} of $W$.  For example, if $W=A_n\iso\fS_{n+1}$, then by equation~\ree{ell} and Proposition~\ref{invmaj} we have
$$
f^\ell(A_n;q)=[2]_q[3]_q\cdots[n+1]_q
$$ 
so that $d_i=i+1$ for $1\le i\le n$.  Table~\ref{degrees} lists the degrees of the irreducible finite Coxeter groups.

\begin{table}
\setlength{\unitlength}{1.5pt}
$$
\barr{c|l}
\mbox{Group}&\mbox{Degrees}\\
\hline
A_n&2,3,4,\ldots,n+1\\
B_n&2,4,6,\ldots,2n\\
D_n&2,4,6\ldots,2n-2,n\\
E_6&2,5,6,8,9,12\\
E_7&2,6,8,10,12,14,18\\
E_8&2,8,12,14,18,20,24,30\\
F_4&2,6,8,12\\
H_3&2,6,10\\
H_4&2,12,20,30\\
I_2(m)&2,m
\earr
$$
\capt{\label{degrees}  The degrees of the irreducible finite Coxeter groups}
\end{table}

There is another description of the degrees which will clarify their name and is valid for complex reflection groups, not just Coxeter groups.  Let $W$ be a finite group acting irreducibly (i.e., having no invariant subspace) on $\bbC^n$ by reflections.  We call $n$ the {\it rank\/} of $W$ and write $\rk W=n$.
 Let $X_n=\{x_1,x_2,\ldots,x_n\}$ be a set of variables.  Then the invariant space $\bbC[X_n]^W$  is a free algebra.  Each algebraically independent set of homogeneous generators has the same set of degrees which we will call $d_1,d_2,\ldots,d_n$.  These $d_i$ are exactly the same as in the Coxeter setting.  Returning to $A_n$ again, we have the natural action on 
$\bbC[X_{n+1}]$.  But each of the reflecting hyperplanes $x_i=x_j$ is perpendicular to the hyperplane $x_1+x_2+\cdots+x_{n+1}=0$, and so to get a space whose dimension is the rank, we need to consider the quotient $\bbC[X_{n+1}]^{A_n}/(x_1+x_2+\cdots+x_{n+1})$.  It is well known that the algebra $\bbC[X_{n+1}]^{A_n}$ is generated freely by the complete homogeneous symmetric polynomials $h_k(X_{n+1})$ where $1\le k\le n+1$.  But in the quotient $h_1( X_{n+1})=x_1+x_2+\cdots+x_{n+1}=0$, so we only need to consider the generators of degrees $2,3,\ldots,n+1$.  

It will also be instructive to see how one can modify the length definition in the complex case.  Let $R$ denote the set of reflections in $W$.  So if $W$ is a finite Coxeter group with generating set $S$ then $R=\{wsw^{-1}\ :\ w\in W,\ s\in S\}$.  
In any finite complex reflection group, define the {\it absolute length of $w\in W$\/}, $\ell_R(w)$, to be the shortest length of an expression $w=r_1r_2\cdots r_k$ with $r_i\in R$ for all $i$.
We have another factorization 
$$
f^{\ell_R}(W;q)=\sum_{w\in W} q^{\ell_R(w)}=\prod_{i=1}^{n} (1+(d_i-1)q)
$$
where the $d_i$ are again the degrees of $W$.

There is one last quantity which we will need to define the analogue of
$\NC(n)$. It is called  the {\it Coxeter number\/}, $h$, of a complex reflection group $W$.  Unfortunately, there are two competing definitions of $h$.  One is to let $h=d_n$, the largest of the degrees.  The other is to set $h=(\# R+\# H)/n$ where $H$ is the set of reflecting hyperplanes of $W$.  But happily these two conditions coincide when $W$  is {\it well generated\/} which means that it has a generating set of reflections of cardinality $\rk W=n$.  Note that this includes the finite Coxeter groups.  Under the well-generated hypothesis, $W$ will also contain a regular element $g$ of order $h$.  As defined by Brady and Watt in the real case~\cite{bw:kag} and Bessis in the complex~\cite{bes:fcr,bes:dbm}, the {\it noncrossing elements\/} in $W$ are
$$
\NC(W)=\{w\in W\ :\ \ell_R(w)+\ell_R(w^{-1}g)=\ell_R(g)\}.
$$
We note that $\ell_R(g)=\rk W=n$.  We will let $\spn{g}$ act on $\NC(W)$ by conjugation.  (One needs to check that this is well defined.)
%  I think you need the fact that for the element g one has ell_R(gw)=ell_R(wg) to get well def.

To see how this relates to $\NC(n)$, map each  element of $\NC(A_{n-1})$ to the partition whose blocks are the cycles of $\pi$ considered as unordered sets.  (A similar idea is behind Lemma~\ref{gM=M}.)  Then this is a bijection with $\NC(n)$.  To illustrate,let  $n=3$ and  $g=(1,2,3)$.  A case-by-case check using the definition yields
$$
\NC(A_2)=\{(1)(2)(3),\ (1,2)(3),\ (1,3)(2),\ (1)(2,3),\ (1,2,3)\}. 
$$
So the image of this set is all partitions of $[3]$ which is $\NC(3)$ since~\ree{acbd} can not be true with only three elements.

The polynomial in the cyclic sieving result will be
$$
\Cat(W;q)=\prod_{i=1}^n \frac{[h+d_i]_q}{[d_i]_q}.
$$
For example
$$
\Cat(A_{n-1};q)=\prod_{i=1}^{n-1} \frac{[n+i+1]_q}{[i+1]_q}=\Cat_n(q).
$$
It can be shown~\cite{bes:fcr,bw:kag} that
$$
\Cat(W;1)=\#\NC(W).
$$
We can now state the Bessis-Reiner result~\cite{br:csn}.
\bth
\label{br}
Let the finite irreducible complex reflection group $W$ be well generated.  Let $g$ be a regular element of order $h$ and let $\spn{g}$ act on $\NC(W)$ by conjugation.  Then
$$
\left(\ \NC(W),\ \spn{g},\ \Cat(W;q)\ \right) 
$$
exhibits the cyclic sieving phenomenon.\hqed
\eth

The Catalan numbers can be generalized to the {\it Fuss-Catalan numbers\/} which are defined by
$$
\Cat_{n,m}=\frac{1}{mn+1}{(m+1)n\choose n}.
$$
Note that $\Cat_{n,1}=\Cat_n$.  The Fuss-Catalan numbers count, among other things, the {\it $m$-divisible\/} noncrossing partitions in $\NC(mn)$, i.e., those which have all their block sizes divisible by $m$.  So, for example, $\Cat_{2,2}={6\choose 2}/5=3$ corresponding to the partitions $\{12,34\}$, $\{14,23\}$, and $\{1234\}$.  (As usual, we are suppressing some set braces and commas.)

A natural $q$-analogue of the Fuss-Catalan numbers for any well-generated finite complex reflection group of rank $n$ is
$$
\Cat^{(m)}(W;q)=\prod_{i=1}^{n}\frac{[mh+d_i]_q}{[d_i]_q}.
$$
Armstrong~\cite{arm:gnp} has constructed a set counted by $\Cat^{(m)}(W;1)$.
As before, consider a regular element $g$ of $W$ having order $h$.  Define 
$$
\NC^{(m)}(W)=\left\{(w_0,w_1,\ldots,w_m)\in W^{m+1}\ \hs{-3.5pt}  :\
\hs{-3.5pt} w_0w_1\cdots w_m=g,\ \sum_{i\ge0}^m \ell_R(w_i)=\ell_R(g)\right\}.
$$

Armstrong also defined two actions of $g$ on $\NC^{(m)}(W)$ and made corresponding cyclic sieving conjectures about them.  These have been proved by Krattenthaler~\cite{kra:npa} for the two infinite 2-parameter families of finite irreducible well-generated complex reflection groups, and by Krattenthaler and M\"uller~\cite{km:csg,km:csgd} for the exceptional ones.
One of the actions is
\beq
\label{action}
g(w_0,w_1,w_2,\ldots,w_m)=(gw_mg^{-1},w_0,w_1,\ldots,w_{m-1}).
\eeq
This generates a group $C_{(m+1)h}$ acting on $\NC^{(m)}(W)$.
\bth
Let the finite irreducible complex reflection group $W$ be well generated.  Let $g$ be a regular element of order $h$ and let $C_{(m+1)h}$ act on $\NC^{(m)}(W)$ by~\ree{action}.  Then
$$
\left(\ \NC^{(m)}(W),\ C_{(m+1)h},\ \Cat^{(m)}(W;q)\ \right) 
$$
exhibits the cyclic sieving phenomenon.\hqed
\eth

We should mention that Gordon and Griffeth~\cite{gg:cnc} have defined a
version of the $q$-Fuss-Catalan polynomials for all complex reflection groups
which specializes to $\Cat^{(m)}(W;q)$ when $W$ is well generated.  The
primary ingredients of their construction are Rouquier's formulation of shift
functors for the rational Cherednik algebras of $W$~\cite{rou:qsa}, and
Opdam's analysis of permutations of the irreducible representations of $W$
arising from the Knizhnik-Zamolodchikov connection~\cite{opd:crg}.
Furthermore, plugging  roots of unity into the Gordon-Griffeth polynomials
yields nonnegative integers.  But finding a corresponding CSP remains
elusive. 

There is another object enumerated by $\Cat(W;1)$ when $W$ is a Coxeter group, namely facets of cluster complexes.  Cluster complexes were introduced by Fomin and Zelevinsky~\cite{fz:yga} motivated by their theory of cluster algebras.
To define them, we need some background on root systems.  Rather than take an axiomatic approach, we will rely on examples and outline the necessary facts we will need.  The reader wishing details can consult the texts of Fulton and Harris~\cite{fh:rt} or Humphreys~\cite{hum:ila}.

To every real reflection group $W$ is associated a {\it root system\/},
$\Phi=\Phi_W$, which consists of a set of vectors called {\it roots\/}
perpendicular to the reflecting hyperplanes.  Each hyperplane has exactly two
roots perpendicular to it and they are negatives of each other.  We require
that $\Phi$ span the space on which $W$ acts.  So, as above, when $W=A_n$ we
have to restrict to the hyperplane $x_1+x_2+\cdots+x_{n+1}=0$.  Finally, $W$
must act on $\Phi$.  To illustrate, if $W=A_n$ then the roots perpendicular to
the hyperplane $x_i=x_j$ are taken to be $\pm(e_i-e_j)$ where $e_i$ is the
$i$th unit coordinate vector.  Let $\Pi=\{\al_1,\al_2,\ldots,\al_n\}$ be a
set of {\it simple roots\/} which correspond to the simple
reflections $s_1,s_2,\ldots,s_n$.  For groups of type $A$-$D$ the standard
choices for simple roots are listed in Table~\ref{roots}.  Note that there are
two root systems $B_n$ and $C_n$ associated with the type $B$ group depending
on whether one considers it as the set of symmetries of a hypercube or a
hyperoctahedron, respectively.   In fact, this group is sometimes refered to
as $BC_n$.

\begin{table}
\setlength{\unitlength}{1.5pt}
$$
\barr{c|l}
\Phi&\Pi\\
\hline
A_n&e_1-e_2,\ e_2-e_3,\ \ldots,\ e_{n-1}-e_n,\ e_n-e_{n+1}\\
B_n&e_1-e_2,\ e_2-e_3,\ \ldots,\ e_{n-1}-e_n,\ e_n\\
C_n&e_1-e_2,\ e_2-e_3,\ \ldots,\ e_{n-1}-e_n,\ 2e_n\\
D_n&e_1-e_2,\ e_2-e_3,\ \ldots,\ e_{n-1}-e_n,\ e_{n-1}+e_n
\earr
$$
\capt{\label{roots}  The simple roots in types $A$-$D$}
\end{table}

One can find a hyperplane $H$ which is not a reflecting hyperplane such that
all $\al_i\in\Pi$  lie on the same side of $H$.  Since the roots come in
opposite pairs, half of them will lie on the same side of $H$ as the simple
roots and these are called the {\it positive roots\/}, $\Phi_{>0}$.  The rest
of the roots are called {\it negative\/}.  The simple roots form a basis for
the span $\spn{\Phi}$ and every positive root can be written as a linear
combination of the simple roots with positive coefficients.  In type $A$ one
can take $H$ to be any plane of the form
$c_1x_1+c_2x_2+\cdots+c_{n+1}x_{n+1}=0$ where $c_1>c_2>\ldots>c_{n+1}>0$.  In
this case $\Phi_{>0}=\{e_i-e_j\ :\ i<j\}$ and we can write 
$$
e_i-e_j=\al_i+\al_{i+1}+\cdots+\al_{j-1}
$$
where $\al_k=e_k-e_{k+1}\in\Pi$ for $1\le k\le n$.

Take the union of the positive roots with the negatives of the simple roots to get
$$
\Phi_{\ge-1}=\Phi_{>0}\cup (-\Pi).
$$
For example, if $\Phi$ is of type $A_2$, then $\Phi_{\ge-1}=\{\al_1,\al_2,\al_1+\al_2,-\al_1,-\al_2\}$ as displayed in Figure~\ref{A_2}.

\bfi
$$
\begin{pspicture}(-2,-2)(4,4)
\SpecialCoor
\psline[arrows=->](0;0)(2;0)
\psline[arrows=->](0;0)(2;60)
\psline[arrows=->](0;0)(2;120)
\psline[arrows=->](0;0)(2;180)
\psline[arrows=->](0;0)(2;300)
\rput(2.5;0){$\alpha_1$}
\rput(2.5;60){$\alpha_1+\alpha_2$}
\rput(2.5;120){$\alpha_2$}
\rput(2.5;180){$-\alpha_1$}
\rput(2.5;300){$-\alpha_2$}
\end{pspicture}
$$
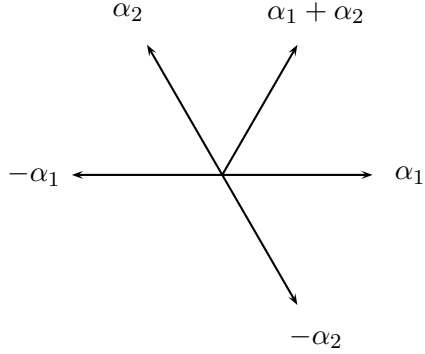
\capt{\label{A_2}  The roots in $\Phi_{\ge-1}$ for type $A_2$}
\efi

Take any partition $[n]=I_+\uplus I_-$ such that the nodes indexed by $I_+$ in
the Dynkin diagram of $W$ are totally disconnected, and the same is true for $I_-$.  Note that this means that the reflections in $I_\ep$ commute for $\ep\in\{+,-\}$.
Next, define a pair of involutions $\tau_{\pm}:\Phi_{\ge-1}\ra\Phi_{\ge-1}$ by
$$
\tau_\ep(\al)=
\case{\al}{if $\al=-\al_i$ for $i\in I_{-\ep}$,}{\left(\dil\prod_{i\in I_\ep} s_i\right)(\al)}{\rule{0pt}{20pt}otherwise.}
$$
Since the roots in $I_\ep$ commute, the product is well defined.
Returning to our example, let $I_+=\{1\}$ and $I_-=\{2\}$,  Table~\ref{tau} displays the images of each root in $\Phi_{\ge-1}$ under $\tau_+$ and $\tau_-$.

\begin{table}
$$
\barr{c|c|c}
\al		&\tau_+(\al)	&\tau_-(\al)\\
\hline
\al_1		&-\al_1		&\al_1+\al_2\\
\al_2		&\al_1+\al_2	&-\al_2\\
\al_1+\al_2	&\al_2		&\al_1\\
-\al_1		&\al_1		&-\al_1\\
-\al_2		&-\al_2		&\al_2
\earr
$$
\capt{\label{tau} The images of $\tau_+$ and $\tau_-$}
\end{table}

Consider the product $\Ga=\tau_-\tau_+$ (where maps are composed right-to-left) and the cyclic group $\spn{\Ga}$ it generates acting on $\Phi_{\ge-1}$.  In the running example, $\Ga$ consists of a single cycle
$$
\Ga=(\ \al_1,\ -\al_1,\ \al_1+\al_2,\ -\al_2,\ \al_2\ ).
$$
This map induces a relation of {\it compatibility\/}, $\al\sim\be$, on $\Phi_{\ge-1}$ defined by the following two conditions.
\ben
\item For $-\al_i\in -\Pi$ and $\be\in\Phi_{>0}$ we have $-\al_i\sim\be$ if and only if $\al_i$ does not occur in the simple root expansion of $\be$.
\item For all $\al,\be\in\Phi_{\ge-1}$ we have $\al\sim\be$ if and only if $\Ga(\al)\sim\Ga(\be)$.
\een
In our example, $-\al_1\sim\al_2$ by the first condition.  Then repeated application of the second yields $\al\sim\be$ for all $\al,\be\in\Phi_{\ge-1}$ when $W=A_2$.

The {\it cluster complex\/}, $\De(\Phi)$, is the abstract simplicial complex (i.e., a family of sets called  {\it faces\/} closed under taking subsets) consisting of all sets of pairwise compatible elements of $\Phi_{\ge0}$.  So in the $A_2$ case, $\De(\Phi)$ consists of a single facet (maximal face) which is all of $\Phi_{\ge-1}$.  The faces of $\De(\Phi)$ can be described in terms of dissections of polygons using noncrossing diagonals.  Using this interpretation, Eu and Fu~\cite{ef:csp} prove the following result.
\bth
\label{ef:thm}
Let $W$ be a finite Coxeter group and let $\Phi$ be the corresponding root system.  Let $\De_{\max}(\Phi)$ be the set of facets of the cluster complex $\De(\Phi)$.  Then the triple
$$
\left(\ \De_{\max}(\Phi),\ \spn{\Ga},\ \Cat(W;q)\ \right)
$$
exhibits the cyclic sieving phenomenon.\hqed
\eth
In fact, Eu and Fu strengthened this theorem in two ways: by looking at the faces of dimension $k$ and by considering an $m$-divisible generalization of the cluster complex due to Fomin and Reading~\cite{fr:gcc}.

\section{A cyclic sieving miscellany}

Here we collect some topics not previously covered.  These include methods for generating new CSPs from old ones, sieving for cyclic polytopes, and extending various results to arbitrary fields.

Berget, Eu, and Reiner~\cite{ber:ccs} gave various ways to construct CSPs and we will discuss one of them now.  Let $f(q)=\sum_{i=0}^l m_i q^i\in\bbN[q]$ satisfy $f(1)=n$.  Let $X_n=\{x_1,x_2,\ldots,x_n\}$ be a set of variables and let $p(X_n)$ be a polynomial symmetric in the $x_i$.  The {\it plethystic substitution\/} of $f$ into $p$ is
$$
p[f]=p(\underbrace{1,\ldots,1}_{m_0},\underbrace{q,\ldots,q}_{m_1},\ldots,\underbrace{q^l,\ldots,q^l}_{m_l}).
$$
Note that since $p$ is symmetric, it does not matter in what order one substitutes these values.  For a concrete example, let $f(q)=1+2q$ and $p(X_3)=h_2(x_1,x_2,x_3)$ as in~\ree{chi'ex}.  Then
$
h_2[f]=h_2(1,q,q)=1+2q+3q^2.
$
If $f(q)=[n]_q$ then $p[f]=p(1,q,\ldots,q^{n-1})$ is the principal specialization of $p$.  Plethysm is useful in describing the representations of wreath products of groups.  

 We will also need the {\it elementary symmetric polynomials\/}
$$
e_k(x_1,x_2,\ldots,x_n)=\sum_{1\le i_1< i_2<\ldots< i_k\le n} x_{i_1}x_{i_2}\cdots x_{i_k},
$$
i.e., the sum of all square-free monomials of degree $k$ in the $x_i$.  For example,
$
e_2(x_1,x_2,x_3)=x_1x_2+x_1x_3+x_2x_3.
$
It is well known that the algebra symmetric polynomials in the variables $X_n$ is freely generated by 
$e_1(X_n),e_2(X_n),\ldots, e_n(X_n)$ and this is sometimes called the Fundamental Theorem of Symmetric Polynomials.  The $e_i(X_n)$ are dual to the $h_j(X_n)$ in a way that can be made precise.  

Finally, it will be convenient to extend some of our concepts slightly.  We will use notations~\ree{(())} and~\ree{()}  replacing $[n]$ with any set.   We can also consider {\it symmetric functions\/} which are formal power series in the variables $X=\{x_1,x_2,x_3,\ldots\}$ which are invariant under all permutations of variables and are of bounded degree.  For example, the complete homogeneous symmetric function is
$$
h_k(X)=\sum_{1\le i_1\le i_2\le\ldots\le i_k} x_{i_1}x_{i_2}\cdots x_{i_k}.
$$
In this case, we can still make a plethystic substitution of a polynomial $f(q)$ by letting $x_i=0$ for $i>n=f(1)$.
One of the results of Berget, Eu, and Reiner~\cite{ber:ccs} is as follows.
\bth
If a triple $(X,C,f(q))$ exhibits the cyclic sieving phenomenon, then the triple
$$
\left(\ \left(\hs{-3pt}{X\choose k}\hs{-3pt}\right),\ C,\ h_k[f(q)]\ \right)
$$
does so as well.

If, in addition, $\#C$ is odd then the triple
$$
\left(\ {X\choose k},\ C,\ e_k[f(q)]\ \right)
$$
also exhibits the cyclic sieving phenomenon.\hqed
\eth

Some remarks about this theorem are in order.  First of all, the authors
actually prove it for $C$ a product of cyclic groups and multi-cyclic sieving,
but this does not materially alter the demonstration.  The proof uses
symmetric tensors for the first part (much in the way they were used in
Section~\ref{rtb}) and alternating tensors in the second.  So the restriction
on $\#C$ is there to control the sign.  Finally, it is instructive to note how
the first part implies our old friend, Theorem~\ref{multiset}.  Indeed, it is
clear that the triple $( [n],\spn{(1,2,\ldots,n)}, [n]_q)$ exhibits the CSP:
Only the identity element of the group has fixed points and there are $n$ of
them, while for $d|n$ we have
$$
[n]_{\om_d}=\case{n}{if $d=1$,}{0}{otherwise.}
$$
Now applying Lemma~\ref{hk:lem} and the previous theorem completes the proof.

Eu, Fu, and Pan~\cite{efp:csp} investigated the CSP for faces of cyclic polytopes.  The {\it moment curve in dimension $d$\/} is  $\ga:\bbR\ra\bbR^d$ defined parametrically by
$$
\ga(t)=(t,t^2,\ldots,t^d).
$$
Given real numbers $t_1<t_2<\ldots<t_n$, the corresponding {\it cyclic polytope\/} is the convex hull
$$
\CP(n,d)=\conv\{\ga(t_1),\ga(t_2),\ldots,\ga(t_n)\}.
$$
A good reference for the theory of convex polytopes is Ziegler's book~\cite{zie:lp}.  

It is known that the $\ga(t_i)$ are the vertices of $\CP(n,d)$.  Also, its  combinatorial type  (i.e., the structure of its faces) does not depend on the parameters $t_i$.  Let $f_k(n,d)$ denote the number of faces of $\CP(n,d)$ of dimension $k$ which is well defined by the previous sentence.  We also let $\CP_k(n,d)$ denote the set of such faces.  Cyclic polytopes are famous, in part, because they have the maximum number of faces in all dimensions $k$ among all polytopes with $n$ vertices in $\bbR^d$.

In what follows, we will assume $d$ is even.  There is a formula for the face
numbers (for all $d$).  In particular, for $0\le k<d$ and $d$ even
$$
f_k(n,d)=\sum_{j=1}^{d/2} \frac{n}{n-j}{n-j\choose j}{j\choose k+1-j}.
$$
The reader will not be surprised that we will use the $q$-analogue
$$
f_k(n,d;q)=\sum_{j=1}^{d/2} \frac{[n]_q}{[n-j]_q}\gau{n-j}{j}_q \gau{j}{k+1-j}_q.
$$

Let $g\in\spn{(1,2,\ldots,n)}$ act on the vertices of $\CP(n,d)$ by sending vertex $\ga(t_i)$ to $\ga(t_{g(i)})$.  For even $d$  this induces an automorphism of $\CP(n,d)$ in that it sends faces to faces.  We can now state the main result of Eu, Fu, and Pan~\cite{efp:csp}.
\bth
Suppose $d$ is even and $0\le k<d$.  Then the triple
$$
\left(\ \CP_k(n,d),\ \spn{(1,2,\ldots,n)},\ f_k(n,d;q)\ \right)
$$
exhibits the cyclic sieving phenomenon.\hqed
\eth
For odd $d$ the $n$-cycle does not necessarily induce an automorphism of $\CP(n,d)$ and so there can be no CSP.  However, in this case there are actions of certain groups of order 2 and it would be interesting to find CSPs for them.

Reiner, Stanton, and Webb~\cite{rsw:sre} considered extending Springer's theory to arbitrary fields.  Broer, Reiner, Smith, and Webb~\cite{brsw:ect} continued this work and also extended various invariant theory results of Chevalley, Shephard-Todd, and Mitchell such as describing the relationship between the coinvariant and group algebras.  Let $V$ be an $n$-dimensional vector space over a field $k$, and let $G$ be a finite subgroup of $GL(V)$ generated by (pseudo)-reflections which are defined as in the complex case.  Then $G$ acts on the polynomial algebra $\cS=k[x_1,x_2,\ldots,x_n]$.  Assume that $\cS^G$ is a (free) polynomial algebra so that $\cS^G=k[f_1,f_2,\ldots,f_n]$ for polynomials $f_1,f_2,\ldots,f_n$.

To define regular elements, one must work in the algebraic closure $\kb$ of
$k$.  Let $\ol{V}=V\otimes_k \kb$.  Call an element $g\in G$ {\it regular\/}
if it has an eigenvector $v\in\ol{V}$ lying on none of the reflecting
hyperplanes $\Hb=H\otimes_k k$ for reflections in $G$.  It can be shown that
in this case $o(g)$ is invertible in $k$.  This implies that the $\spn{g}$-submodules of
any $G$-module are {\it completely reducible\/}, meaning that they can be written as a
direct sum of irreducibles.  (Recall that complete reducibility 
is not guaranteed over arbitrary fields as it is over $\bbC$ by Maschke's
Theorem, Theorem~\ref{G} (b).) 

Now consider any subgroup $H\le G$.  The cyclic sieving set will be the cosets
$G/H$ acted upon by left multiplication of the regular element $g$.    For the function we will take
the quotient $\Hilb(\cS^H;q)/\Hilb(\cS^G;q)$.  But one has to make sure that this
is in $\bbN[q]$ and not just a rational function.  Reiner, Stanton, and
Webb~\cite{rsw:sre} explained why this must be a polynomial with integer
coefficients.  But in stating their  CSP they had to assume extra conditions
on $H$ so that they could prove the coefficients were nonnegative.  They also
asked whether it was possible to prove the CSP without these hypotheses, and
this was done in by Broer, Reiner, Smith, and Webb~\cite{brsw:ect} thus
generalizing Theorem~\ref{reg}. 
\bth
\label{brswthm}
Let $V$ be a finite-dimensional vector space over a field $k$.  Let $G$ be a finite subgroup of $\GL(V)$ for which $\cS^G$ is a polynomial algebra.   Let $g$ be a regular element of $G$ acting on $G/H$ by left multiplication.  Then for any $H\le G$, the triple
$$
\left(\ G/H,\ \spn{g},\ \frac{\Hilb(\cS^H;q)}{\Hilb(\cS^G;q)}\ \right)
$$
exhibits the cyclic sieving phenomenon.\hqed
\eth

\section{Remarks}

\subsection{Alternate definitions}

In their initial paper~\cite{rsw:csp}, Reiner, Stanton, and White gave a
second, equivalent,  definition of the CSP.  While this one has not come to be used as much
as~\ree{csp}, we mention it here for completeness.   

Given a group $G$ acting on a set $X$, denote the stabilizer subgroup of $y\in X$ by
$$
G_y=\{g\in G\ :\ gy=y\}.
$$
If $x,y$ are in the same orbit $\cO$ then their stabilizers are conjugate (and if $G$ is Abelian they are actually equal).  So if $y\in\cO$ then call  $s(\cO)=\#G_y$ the {\it stabilizer-order\/} of $\cO$ which is well defined by the previous sentence.

Suppose $f(q)=\sum_{i\ge0} m_i q^i\in\bbN[q]$ and define coefficients $a_i$ for $0\le i<n$ by
$$
f(q)\Cong a_0+a_1q+\cdots+a_{n-1} q^{n-1}\ (\Mod 1-q^n).
$$
Equivalently,
$$
a_i=\sum_{j\Cong i\ (\Mod n)} m_j.
$$
\begin{Def}
Suppose $\# X=n$, $C$ is a cyclic group acting on $X$, and the $a_i$ are as above.  The triple $(X,C,f(q))$ exhibits the {\it cyclic sieving phenomenon\/} if, for  $0\le i<n$,
\beq
\label{csp2}
a_i=\#\{\cO\ :\ s(\cO)|i\}.
\eeq
\end{Def}

Note that if~\ree{csp2} is true then $a_0$ counts the total number of $C$-orbits, while $a_1$ counts the number of free orbits (those of size $\# C$).  In fact, one can use these equations to determine how many orbits there are of any size using M\"obius inversion.  Returning to our original example with $C=\spn{(1,2,3)}$ and  $2$-element multisets on $[3]$, the orbits were
$$
\cO_1=(11,22,33),\ \cO_2=(12,23,13).
$$
On the other hand
$$
f(q)=1+q+2q^2+q^3+q^4\Cong 2+2q+2q^2\ (\Mod 1-q^3)
$$
indicating that there are 2 orbits total with both of them being free.  The coefficient $a_2=2$ as well since a free orbit's stabilizer-order of 1 will divide any other.

The proof that these two definitions are equivalent is via the representation theory paradigm, Theorem~\ref{csp:thm}.  The main tool is Frobenius reciprocity.

In a personal communication, Reiner has pointed out that it might also be
interesting to define cyclic sieving with more general polynomials.  One
possibility would be to allow negative integral coefficients which could be useful,
for example, when considering quotients of Hilbert series.  Note that
this issue arose in the genesis of Theorem~\ref{brswthm}.
Another extension could be to Laurent polynomials, that is, elements of
$\bbZ[q,q^{-1}]$.  Such polynomials might come up when considering a bivariate
generating function $f(q,t)$ where one lets $t=q^{-1}$.  We have
already seen such a substitution in Theorem~\ref{S(la)}, although in that case
it turns out that the generating function remains an ordinary polynomial.

\subsection{More on Catalan CSPs}

We have just begun to scratch the surface of the connection between Catalan combinatorics and cyclic sieving.  We have already mentioned how polygonal dissections is behind the work of Eu and Fu on cluster complexes~\cite{ef:csp} as in  Theorem~\ref{ef:thm}.  We will now describe an open problem and some ongoing work about triangulations, i.e., dissections where every face is a triangle.

Let $P$ be a regular $n$-gon.  Let $\cT_n$ denote the set of triangulations $T$ of $P$ using nonintersecting diagonals.  It is well known that
$$
\#\cT_{n+2}=\Cat_n.
$$ 
We will act on triangulations by clockwise rotation.  So, for example, for the pentagon these is  only one cycle
$$
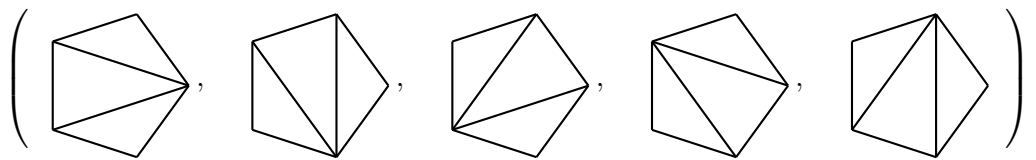
\left(\
\begin{pspicture}(-1,0)(1.1,1)
\SpecialCoor
\psline(1;0)(1;72)
\psline(1;72)(1;144)
\psline(1;144)(1;216)
\psline(1;216)(1;288)
\psline(1;288)(1;0)
\psline(1;0)(1;144)
\psline(1;0)(1;216)
\end{pspicture}
,\quad
\begin{pspicture}(-1,0)(1.1,1)
\SpecialCoor
\psline(1;0)(1;72)
\psline(1;72)(1;144)
\psline(1;144)(1;216)
\psline(1;216)(1;288)
\psline(1;288)(1;0)
\psline(1;288)(1;72)
\psline(1;288)(1;144)
\end{pspicture}
,\quad
\begin{pspicture}(-1,0)(1.1,1)
\SpecialCoor
\psline(1;0)(1;72)
\psline(1;72)(1;144)
\psline(1;144)(1;216)
\psline(1;216)(1;288)
\psline(1;288)(1;0)
\psline(1;216)(1;0)
\psline(1;216)(1;72)
\end{pspicture}
,\quad
\begin{pspicture}(-1,0)(1.1,1)
\SpecialCoor
\psline(1;0)(1;72)
\psline(1;72)(1;144)
\psline(1;144)(1;216)
\psline(1;216)(1;288)
\psline(1;288)(1;0)
\psline(1;144)(1;288)
\psline(1;144)(1;0)
\end{pspicture}
,\quad
\begin{pspicture}(-1,0)(1.1,1)
\SpecialCoor
\psline(1;0)(1;72)
\psline(1;72)(1;144)
\psline(1;144)(1;216)
\psline(1;216)(1;288)
\psline(1;288)(1;0)
\psline(1;72)(1;216)
\psline(1;72)(1;288)
\end{pspicture}\
\right).
$$

Reiner, Stanton, and White~\cite{rsw:csp} proved the following theorem in this setting.  (In fact, they proved a stronger result about dissections using noncrossing diagonals where one fixes the number of diagonals.)
\bth
\label{rsw:cat}
Let $C_{n+2}$ act on $\cT_{n+2}$ by rotation.  Then the triple
$$
(\ \cT_{n+2},\ C_{n+2},\ \Cat_n(q)\ )
$$
exhibits the cyclic sieving phenomenon.\hqed
\eth

Their proof was of the sort where one evaluates both sides of~\ree{csp} directly.  But as mentioned before, these proofs often lack the beautiful insights one obtains from using representation theory.  It would be very interesting to find such a proof, perhaps by finding an appropriate complex reflection group along with a basis and fake degree polynomial which would permit the use of Westbury's Theorem~\ref{wes:thm}.

\bfi
$$
\begin{pspicture}(-2,-1.5)(3.1,2)
\SpecialCoor
\psline(1;0)(1;72)
\psline(1;72)(1;144)
\psline(1;144)(1;216)
\psline(1;216)(1;288)
\psline(1;288)(1;0)
\psline(1;0)(1;144)
\psline(1;0)(1;216)
\rput(1.3;0){$1$}
\rput(1.3;288){$2$}
\rput(1.3;216){$1$}
\rput(1.3;144){$2$}
\rput(1.3;72){$1$}
\end{pspicture}
\hspace{30pt}
\begin{pspicture}(-2,-1.5)(3.1,2)
\SpecialCoor
\psline(1;0)(1;72)
\psline(1;72)(1;144)
\psline(1;144)(1;216)
\psline(1;216)(1;288)
\psline(1;288)(1;0)
\psline(1;216)(1;0)
\psline(1;216)(1;72)
\rput(1.3;0){$1$}
\rput(1.3;288){$2$}
\rput(1.3;216){$1$}
\rput(1.3;144){$2$}
\rput(1.3;72){$1$}
\end{pspicture}
$$
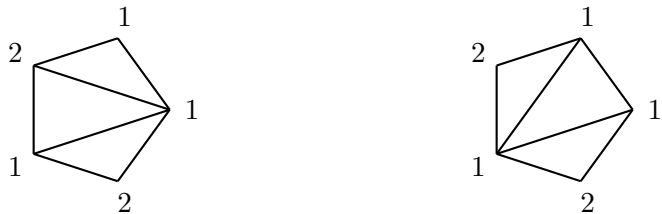
\capt{\label{pro} Two triangulations, one proper (left) and one not (right)}
\efi

One can also consider colored triangulations.  Label (``color'') the vertices of the polygon $P$ clockwise $1,2,1,2,\ldots$.  (When $n$ is odd, there will be an edge of $P$ with both endpoints labeled $1$.)  Call a triangulation {\it proper\/} if it contains no monochromatic triangle.  (This terminology is both by analogy with proper coloring of graphs and in honor of Jim Propp who first conjectured~\ree{cP}.)  In Figure~\ref{pro}, the left-hand triangulation is proper while the one on the right is not.  Let $\cP_n$ be the set of proper triangulations of an $n$-gon.  Sagan~\cite{sag:ppp} proved that
\beq
\label{cP}
\# \cP_{N+2}=\case{\dil\frac{2^n}{2n+1}{3n\choose n}}
{if $N=2n$ where $n\in\bbN$,}
{\dil\frac{2^{n+1}}{2n+2}{3n+1\choose n}}
{if $N=2n+1$ where $n\in\bbN$.\rule{0pt}{30pt}}
\eeq
Note that for $N=2n$ we have $\# \cP_{N+2}= 2^n \Cat_{2,n}$.

Roichman and Sagan~\cite{rs:ccp} are studying CSPs for colored dissections.  In the triangulation case, notice that when $n$ is odd then there is no action of $C_n$ on $\cP_n$ because it is possible for the rotation of a proper triangulation to be improper as in Figure~\ref{pro}.  So it only makes sense to consider a rotational CSP for $n$ even.  They have proved that one does indeed have such a phenomenon, although the necessary $q$-analogue for the factor of $2^n$ is somewhat surprising.
\bth
Let $N=2n$ and let $C_{N+2}$ act on $\cP_{N+2}$ by rotation.  Then the triple
$$
\left(\ \cP_{N+2},\ C_{N+2},\ 
\frac{[2]_{q^2}\left([2]_q^{n-1}-[2]_q^{\ce{n/2}-1}+2^{\ce{n/2}-1}\right)}{[2n+1]_q}
\gau{3n}{n}_q\ \right)
$$
exhibits the cyclic sieving phenomenon.\hqed
\eth

\subsection{A combinatorial proof}

One could hope for purely combinatorial proofs of CSPs.  Since it may not be clear exactly what this would entail, consider the following paradigm.  First of all, we would need to have a combinatorial expression for $f(q)$, namely some statistic on the set $X$ such that
\beq
\label{f(X)}
f^{\sta}(X;q)=f(q).
\eeq

Suppose further that, for each $g\in C$, one has a partition of $X$
$$
\pi=\pi_g=\{B_1,B_2,\ldots,B_k\}
$$
satisfying the following criterion where $\om=\om_{o(g)}$:
\beq
\label{f(B_i)}
f^{\sta}(B_i;\om)=\case{1}{if $i\le\# X^g$,}{0}{if $i>\# X^g$.}
\eeq
In other words, the initial blocks correspond to the fixed points of $g$ and their weights evaluate to 1 when plugging in $\om$, while the weights of the rest of the blocks get zeroed out under this substitution.  
(In practice, the $B_i$ for $i\le\# X^g$ are singletons each with weight $q^j$
where $o(g)|j$, while for $i>\# X^g$ the sum of the weights in the block form
a geometric progression which becomes zero since $1+\om+\cdots+\om^{d-1}=0$
for any proper divisor $d$ of $o(g)$.)
In this case, one automatically has cyclic sieving because, using equations~\ree{f(X)} and~\ree{f(B_i)} as well as the fact that we have a partition
$$
f(\om)=f^{\sta}(X;\om)=\sum_{i\ge0} f^{\sta}(B_i;\om)=
\underbrace{1+\cdots+1}_{\# X^g}+0+\cdots+0=\# X^g.
$$

Roichman and Sagan~\cite{rs:ccp} have succeeded in using this method to prove
Theorem~\ref{rsw:first}.  They are currently working on trying to apply it to
various other cyclic sieving results.

\vs{.25in}

\noindent{{\bf Note added in proof}}

\vs{.18in}

Since this article was written six new papers have appeared related to the
cyclic sieving phenomenon.  For completeness' sake, we briefly describe each of them
here.

In~\cite{ast:ubb}, Armstrong, Stump, and Thomas constructed a bijection between
noncrossing partitions and nonnesting partitions which sends
a complementation map of Kreweras~\cite{kre:pnc}, {\bf Krew}, to a function of
Panyushev~\cite{pan:oap},
{\bf Pan}.  Using this construction they are able to prove two CSP
conjectures in the paper of Bessis and Reiner~\cite{br:csn} about {\bf Krew}
and {\bf Pan}.
The first refines Theorem~\ref{br} in the case that $W$ is a finite Coxeter
group since ${\bf Krew}^2$ coincides with the action of the regular element.
The second follows from the first using the bijection.

Kluge and Rubey~\cite{kr:cst} have obtained a cyclic sieving result for
rotation of 
Ptolemy diagrams.  These diagrams were recently introduced in a paper of Holm,
J{\o}rgensen, and Rubey~\cite{hjr:pdt} as a model for torsion pairs in the
cluster category of type $A$.  Their result is related to
the generalization of Theorem~\ref{rsw:cat} mentioned just before its
statement, except that certain diagonals are allowed to cross and  
one keeps track of the number of regions of various types rather than the
number of diagonals.  But the polynomial in both cases is a product of
$q$-binomial coefficients, so it would be interesting to find a common
generalization.

Noncrossing graphs on the vertex set $[n]$ can be defined analogously to
noncrossing polygonal dissections by arranging the vertices around a circle and insisting
that the resulting graph be planar.  Flajolet and Noy~\cite{fn:acn} showed
that the number of connected noncrossing graphs with $n$ vertices and $k$
edges is given by
$$
\frac{1}{n-1}{3n-3\choose n+k}{k-1\choose n-2}
$$
Following a personal communication of S.-P.\ Eu, Guo~\cite{guo:csp} has shown
that one has a CSP using these graphs, rotation, and the expected $q$-analogue
of the expression above.

Another way to generalize noncrossing dissections into triangles is to
define a {\it $k$-triangulation\/} of a convex $n$-gon to be a maximal
collection of diagonals such that no $k+1$ of them mutually cross.  So
ordinary triangulations are the case $k=1$.  In a personal communication,
Reiner has conjectured a CSP for such triangulations under rotation
generalizing Theorem~\ref{rsw:cat}.  In~\cite{ss:mfm}, Serrano and Stump
reformulated this conjecture in terms of $k$-flagged tableaux (certain
semistandard tableaux with bounds on the entries).  But the conjecture remains
open.  

As has already been mentioned, there is no representation theory proof of
Theorem~\ref{rsw:cat}.  The same is true of the Eu and Fu's result,
Theorem~\ref{ef:thm}.  In an attempt to partially remedy this situation,
Rhoades~\cite{rho:csc} has used representation theory and cluster
multicomplexes to prove related CSPs.  
His tools include a notion of
noncrossing tableaux due to Pylyavskyy~\cite{pyl:nct} and geometric
realizations of finite type cluster algebras due to Fomin and Zelevinsky~\cite{fz:ca2}.

Westbury~\cite{wes:itc2} has succeeded in generalizing Rhoades's original
result~\cite{rho:csp} just as he was able to do for the special case
considered by Petersen, Pylyavskyy, and Rhoades~\cite{ppr:pcs} for two and
three rows.  It turns out that the same tools (crystal bases, based modules,
and regular elements) can be used.

\myaddress

\end{document}